\documentclass[12pt,a4paper]{article}%
\usepackage{amsmath}
\usepackage{amssymb}
\usepackage{makeidx}
\usepackage{amsfonts}
\usepackage{hyperref}
\usepackage{graphicx}%
\providecommand{\U}[1]{\protect\rule{.1in}{.1in}}
\newtheorem{theorem}{Theorem}[section]

\newtheorem{corollary}[theorem]{Corollary}
\newtheorem{definition}[theorem]{Definition}
\newtheorem{example}[theorem]{Example}

\newtheorem{remark}[theorem]{Remark}

\newenvironment{proof}[1][Proof]{\noindent\textbf{#1.} }{\rule{0.5em}{0.5em}}

\begin{document}

\title{Hierarchical Schr\"{o}dinger-type operators: the case of potentials with local
singularities }
\author{Alexander Bendikov\thanks{A. Bendikov was supported by the Polish National
Science center, grant 2015/17/B/ST1/00062 and by SFB 1283 of the German
Research Council. }
\and Alexander Grigor'yan\thanks{A. Grigor'yan was supported by SFB 1283 of the
German Research Council. }
\and Stanislav Molchanov\thanks{S. Molchanov was supported by the Russian Science
Foundation (Projects: 20-11-20119 and 17-11-01098). }}
\maketitle

\begin{abstract}
The goal of this paper is twofold. We prove that the operator $H=L+V$ , a
perturbation of the Taibleson-Vladimirov multiplier $L=\mathfrak{D}^{\alpha}$
by a potential $V(x)=b\left\Vert x\right\Vert ^{-\alpha},$ $b\geq b_{\ast},$
is essentially self-adjoint and non-negative definite (the critical value
$b_{\ast}$ depends on $\alpha$ and will be specified later). While the
operator $H$ is non-negative definite the potential $V(x)$ may well take
negative values, e.g. $b_{\ast}<0$ for all $0<\alpha<1$. The equation $Hu=v$
admiits a Green function $g_{H}(x,y)$, the integral kernel of the operator
$H^{-1}$. We obtain sharp lower- and upper bounds on the ratio of the
functions $g_{H}(x,y)$ and $g_{L}(x,y)$. Examples illustrate our exposition.

\end{abstract}
\tableofcontents

\section{Introduction}

\setcounter{equation}{0}

The spectral theory of nested fractals similar to the Sierpinski gasket, i.e.
the spectral theory of the corresponding Laplacians, is well understood. It
has several important features: Cantor-like structure of the essential
spectrum and, as result, the large number of spectral gaps, presence of
infinite number of eigenvalues each of which has infinite multiplicity and
compactly supported eigenstates, non-regularly varying at infinity heat
kernels which contain oscilated in $\log t$ scale terms etc, see
\cite{GrabnerWoess}, \cite{DerfelGrabner} and \cite{BCW}.

The spectral properties mentioned above occure in the very precise form for
the Taibleson-Vladimirov Laplacian $\mathfrak{D}^{\alpha}$, the operator of
fractional derivative of order $\alpha$. This operator can be introduced in
several different forms (say, as $L^{2}$-multiplier in the $p$-adic analysis
setting, see \cite{Vladimirov}) but we select the geometric approach
\cite{Dyson1}, \cite{Molchanov}, \cite{Molchanov1}, \cite{BGP}, \cite{BGPW},
\cite{BendikovKrupski} and \cite{BGMS}.

\subsection{The Dyson hierarchical model}

Let us fix an integer $p\geq2$ and consider the family of partitions
$\{\Pi_{r}:r\in\mathbb{Z}\}$ of the set $X=\mathbb{[}0,+\infty)$ such that
each $\Pi_{r}$ consists of all $p$-adic intervals $I=[kp^{r},(k+1)p^{r})$. We
call $r$ the rank of the partition $\Pi_{r}$\ (respectively, the rank of the
interval $I\in\Pi_{r}$). Each interval of rank $r$ is the union of $p$
disjoint intervals of rank $(r-1)$. Each point $x\in X$ \ belongs to a certain
interval $I_{r}(x)$ of rank $r$, and the intersection of all $p$-adic
intervals $I_{r}(x)$ is $\{x\}.$

\emph{The hierarchical distance} $d(x,y)$ is defined as the Lebesgue measure
$l(I)$ of the minimal $p$-adic interval $I$ which contains $x$ and $y$. Since
any two points $x$ and $y$ belong to a certain $p$-adic interval,
$d(x,y)<\infty$. Clearly $d(x,y)=0$ if and only if $x=y$, $d(x,y)=d(y,x)$.
Moreover, for arbitrary $x,y$ and $z$ the \emph{ultrametric inequality }(which
is stronger than the triangle inequality) holds
\begin{equation}
d(x,y)\leq\max\{d(x,z),d(z,y)\}. \label{UMI}%
\end{equation}
The ultrametric space $(X,d)$ is \emph{complete, separable} and \emph{proper
}metric space. In $(X,d)$ the set of all open balls is countable and coincides
with the set of all $p$-adic intervals. In particular, any two balls either do
not intersect or one is a subset of another. Thus $(X,d)$ is a totally
disconnected separable topological space whence it is homeomorphic to the
Cantor set with a punctured point $\{0,1\}^{\aleph_{0}}\setminus\{o\}$. The
Borel $\sigma$-algebra generated by the ultrametric balls coincides with the
Borel $\sigma$-algebra generated by the Eucledian balls.

\emph{The hierarchical Laplacian} $L$ is defined pointwise as
\begin{equation}
(Lf)(x)=%
%TCIMACRO{\dsum \limits_{r=-\infty}^{+\infty}}%
%BeginExpansion
{\displaystyle\sum\limits_{r=-\infty}^{+\infty}}
%EndExpansion
C(r)\left(  f(x)-\frac{1}{l(I_{r}(x))}%
%TCIMACRO{\dint \limits_{I_{r}(x)}}%
%BeginExpansion
{\displaystyle\int\limits_{I_{r}(x)}}
%EndExpansion
fdl\right)  ,\text{ \ } \label{HL}%
\end{equation}
where $C(r)=(1-\kappa)\kappa^{r-1},$ $r\in\mathbb{Z}$. The series in
(\ref{HL}) diverges in general but it is finite and belongs to $L^{2}(X,l)$
for any $f$ \ which has compact support and takes constant values on the
$p$-adic intervals of a fixed rank $r$. The set of such functions we denote by
$\mathcal{D}$ and call it the set of test functions.

The operator $L$ admits a complete system of compactly supported
eigenfunctions. Indeed, let $I$ be a $p$-adic interval of rank $r$, and
$I_{1},I_{2},...,I_{p}$ be its $p$-adic subintervals of rank $r-1$. Let us
consider $p$ functions%
\[
\psi_{I_{i}}=\frac{1_{I_{i}}}{l(I_{i})}-\frac{1_{I}}{l(I)}.
\]
Each function $\psi_{I_{i}}$\ belongs to $\mathcal{D}$ and satisfies%
\[
L\psi_{I_{i}}=\kappa^{r-1}\psi_{I_{i}}\text{. \ }%
\]
When $I$ runs over the set of all $p$-adic intervals the set of eigenfunctions
$\psi_{I_{i}}$ is complete in $L^{2}(X,l)$. In particular, $L$ is a
essentially self-adjoint operator with pure point spectrum
\[
Spec(L)=\{0\}\cup\{\kappa^{r}:r\in\mathbb{Z}\}.
\]
Clearly each eigenvalue $\lambda(I)=\kappa^{r-1}$ has infinite multiplicity.
In particular, the spectrum of $L$ coincides with its essential part.

We shell see below that writing $\kappa=p^{-\alpha}$ the operator $L$
coinsides with the Taibleson-Vladimirov operator $\mathfrak{D}^{\alpha}$, the
operator of fractional derivative of order $\alpha$. The constant $D=2/\alpha$
is called the $\emph{spectral}$ \emph{dimension }related\emph{ }to the
operator\emph{ }$L$.

According to \cite{BGPW} the operator $L:\mathcal{D}\rightarrow L^{2}(X,l)$
can be represented as a hypersingular integral operator whose integral kernel
$J(x,y)$ is isotropic, i.e. $J(x,y)$ is a function of the distance
$\mathrm{d}(x,y)$, namely we have
\[
Lf(x)=%
%TCIMACRO{\dint \limits_{0}^{\infty}}%
%BeginExpansion
{\displaystyle\int\limits_{0}^{\infty}}
%EndExpansion
\left(  f(x)-f(y)\right)  J(x,y)dl(y)
\]
and%
\[
J(x,y)=\frac{\kappa^{-1}-1}{1-\kappa p^{-1}}\cdot\frac{1}{\mathrm{d}%
(x,y)^{1+2/D}}.
\]
The Markovian semigroup $(e^{-tL})_{t>0}$ admits a continuous transition
density $p(t,x,y)$. The spectral dimension $D=2/\alpha$ indicates the
asymptotic behavior of the function $p(t,x,y)$, e.g.
\[
p(t,x,x)=t^{-D/2}\mathcal{A(}\log_{p}t\mathcal{)},
\]
where $\mathcal{A(\tau)}$ is a continuous non-constant $\alpha$-periodic
function, see \cite[Proposition 2.3]{MolchanovVainberg1}, \cite{BGPW} and
\cite{BCW}.

There are already several publications on the spectrum of the hierarchical
Laplacian acting on a general ultrametric measure space $(X,d,m)$
\cite{AlbeverioKarwowski}, \cite{AisenmanMolchanov}, \cite{Molchanov},
\cite{Molchanov1}, \cite{BGP}, \cite{BGPW}, \cite{BendikovKrupski},
\cite{BGMS}. Accordingly, the hierarchical Schr\"{o}dinger operator was
studied in \cite{Dyson2}, \cite{Molchanov}, \cite{MolchanovVainberg1},
\cite{MolchanovVainberg2}, \cite{Bovier}, \cite{Kvitchevski1},
\cite{Kvitchevski2}, \cite{Kvitchevski3} (the hierarchical lattice of Dyson)
and in \cite{Vladimirov94}, \cite{VladimirovVolovich}, \cite{Kochubey2004}
(the field of $p$-adic numbers).

By the general theory developed in \cite{BGP}, \cite{BGPW} and
\cite{BendikovKrupski}, any hierarchical Laplacian $L$ acts in $L^{2}(X,m),$
is essentially self-adjoint and can be represented as a hypersingular integral
operator
\begin{equation}
Lf(x)=%
%TCIMACRO{\dint \limits_{X}}%
%BeginExpansion
{\displaystyle\int\limits_{X}}
%EndExpansion
(f(x)-f(y))J(x,y)dm(y)\text{. \ } \label{Spectrum}%
\end{equation}
The operator $L$ has a pure point spectrum, its Markovian semigroup
$(e^{-tL})_{t>0}$ admits with respect to $m$ a continuous transition density
$p(t,x,y)$. In terms of certain (intrinsically related to $L$) ultrametric
$d_{\ast}(x,y)$ the functions $J(x,y)$ and $p(t,x,y)$ can be represented in
the form%
\begin{equation}
\text{\ }J(x,y)=\int\limits_{0}^{1/\emph{d}_{\ast}(x,y)}N(x,\tau)d\tau\text{,}
\label{Jump-kernel}%
\end{equation}%
\begin{equation}
\emph{p}(t,x,y)=t\int\limits_{0}^{1/\emph{d}_{\ast}(x,y)}N(x,\tau)\exp
(-t\tau)d\tau. \label{d*-jump kernel}%
\end{equation}
The function $N(x,\tau)$ is called the\emph{ spectral function} and will be
specified later.

\subsection{Outline}

Let us describe the main body of the paper. In Section 2 we introduce the
notion of homogeneous hierarchical Laplacian $L$ and list its basic properties
e.g. the spectrum of the operator\emph{ }$L$ is pure point, all eigenvalues of
$L$ have infinite multiplicity and compactly supported eigenfunctions, the
heat kernel $p(t,x,y)$ exists and is a continuous function having certain
asymptotic properties etc. For the basic facts related to the ultrametric
analysis of heat kernels listed here we refere to \cite{BGPW},
\cite{BendikovKrupski}.

As a special example we consider the case $X=\mathbb{Q}_{p},$ the ring of
$p$-adic numbers endowed with its standard ultrametric $d(x,y)=\left\Vert
x-y\right\Vert _{p}$ and the normed Haar measure $m$. The hierarchical
Laplacian $L$ in our example coincides with the Taibleson-Vladimirov operator
$\mathfrak{D}^{\alpha}$, the operator of fractional derivative of order
$\alpha$, see \cite{Vladimirov}, \cite{Vladimirov94}, and \cite{Kochubey2004}.
The most complete sourse for the basic definitions and facts related to the
$p$-adic analysis is \cite{Koblitz} and \cite{Taibleson75}.

In the next sections we consider the Schr\"{o}dinger-type operator
$H=\mathfrak{D}^{\alpha}+V$ with potential $V\in L_{loc}^{1}$ having local
singularity, e.g. $V(x)=b\left\Vert x\right\Vert _{p}^{-\alpha}$, $0<\alpha
<1$. The main aim here is to prove that the the symmetric operator $H$ defined
(via quadratic forms) on $\mathcal{D}$, the set of locally constant functions
with compact supports, is semibounded and whence admits a self-adjoint
extension. Under certain conditions on $V$ we will prove that $H$ is
essentially self-adjoint operator.

We also prove several results about the negative part of the spectrum of $H$.
For instence, if $V\in L^{p}$ for some $p>1/\alpha$, then the operator $H$ has
essential spectrum equals to the spectrum of $\mathfrak{D}^{\alpha}$. In
particular, if $H$ has any negative spectrum, then it consists of a sequence
of negative eigenvalues of finite multiplicity. If this sequence is infinite
then it converges to zero.

In the concluding section we consider the operator $H=\mathfrak{D}^{\alpha
}+b\left\Vert x\right\Vert _{p}^{-\alpha}$ assuming that $0<\alpha<1$ and
$b\geq b_{\ast}$, the critical value which will be specified later. We prove
that the equation $Hu=v$ admits a fundamenthal solution $g_{H}(x,y)$ (the
Green function of the operator $H$). The function $g_{H}(x,y)$ is continuous
and takes finite values off the diagonal. Let $g_{\mathfrak{D}^{\alpha}}(x,y)$
be the Green function of the operator $\mathfrak{D}^{\alpha}$. The main result
of this section is the following statement: for any $b\geq b_{\ast}$ there
exists $\frac{\alpha-1}{2}\leq\beta<\alpha$ such that
\[
\frac{g_{H}(x,y)}{g_{\mathfrak{D}^{\alpha}}(x,y)}\asymp\left(  \frac
{\left\Vert x\right\Vert _{p}}{\left\Vert y\right\Vert _{p}}\wedge
\frac{\left\Vert y\right\Vert _{p}}{\left\Vert x\right\Vert _{p}}\right)
^{\beta},
\]
where the sign $\asymp$ means that the ratio of the left- and right hand sides
is bounded from below and above by positive constants. This result must be
compared with the Green function estimates for Schr\"{o}dinger operators on
complete Riemanian manifolds, see \cite{Grigoryan}.

\section{Preliminaries}

\setcounter{equation}{0}

\subsection{Homogeneous ultrametric space}

Let $(X,d)$ be a locally compact and separable ultrametric space. Recall that
a metric $d$ is called an \emph{ultrametric} if it satisfies the ultrametric
inequality
\begin{equation}
d(x,y)\leq\max\{d(x,z),d(z,y)\},
\end{equation}
that is stronger than the usual triangle inequality. The basic consequence of
the ultrametric property is that each open ball is a closed set. Moreover,
each point $x$ of a ball $B$ can be regarded as its center, any two balls $A$
and $B$ either do not intersect or one is a subset of another etc. In
particular, the ultrametric space $(X,d)$ is totally disconnected. See e.g.
Section 1 in \cite{BendikovKrupski} and references therein. In this paper we
assume that the ultrametric space $(X,d)$ is not compact and that it is
\emph{proper}, i.e. each closed ball is a compact set.

Let $\mathcal{B}$ be the set of all open balls and $\mathcal{B}(x)\subset
\mathcal{B}$ the set of all balls centred at $x$. Notice that the set
$\mathcal{B}$ is a countable set whereas $X$ by itself may well be
uncountable, e.g. $X=[0,+\infty)$ with $\mathcal{B}$ consisting of all
$p$-adic intervals.

To any ultrametric space $(X,d)$ one can associate in a standard fashion a
tree $\mathcal{T}.$ The vertices of the tree are metric balls, the boundary
$\partial\mathcal{T}$ can be identified with the one-point compactification
$X\cup\{\varpi\}$ of $X.$ We refere to \cite{BendikovKrupski} for a treatment
of the association between an ultrametric space and the tree of its metric balls.

An ultrametric measure space $(X,d,m)$ is called \emph{homogeneous} if the
group of isometries of $(X,d)$ acts transitively and preserves the measure. In
particular,\ a homogeneous ultrametric measure space is eather discrete or
perfect. In a homogeneous ultrametric measure space any two balls $A$ and $B$
having the same diameter satisfy $m(A)=m(B)$. Furthermore, one can choose an
ultrametric generating the same set of balls $\mathcal{B}$ and such that
$m(B)=\mathrm{diam}(B)$ for any ball $B$.

It is remarkable but easy to proof that $X$ can be identified with certain
locally compact Abelian group equipped with translation invariant ultrametric
$d$ and Haar measure $m$. This identification is not unique. One possible way
to define such identification is to choose the sequence $\{a_{n}\}$\ of
forward degrees associated with the tree of balls $\mathcal{T}$. This sequence
is two-sided if $X$ is non-compact and perfect, it is one-sided if $X$ is
compact and perfect, or if $X$ is discrete. In the 1st case we identify $X$
with $\Omega_{a}$, the ring of $a$-adic numbers, in the 2nd case with
$\Delta_{a}\subset\Omega_{a}$, the ring of $a$-adic integers, and in the 3rd
case with the discrete group $\Omega_{a}/\Delta_{a}$. We refere the reader to
the monograph \cite[(10.1)-(10.11), (25.1)-(25.2)]{HewittRoss} for the
comprehensive treatment of special groups $\Omega_{a}$, $\Delta_{a}$ and
$\Omega_{a}/\Delta_{a}$.

\subsection{Homogeneous hierarchical Laplacian}

\ Let $(X,d,m)$ be a non-compact homogeneous ultrametric measure space. Let
$C:\mathcal{B}\rightarrow(0,\infty)$ be a function satisfying the following conditions:

\begin{description}
\item[(i)] $C(A)=C(B)$ for any two balls $A$ and $B$ of the same diameter,

\item[(ii)] $\lambda(B):=\sum\limits_{T\in\mathcal{B}:\text{ }B\subseteq
T}C(T)<\infty$ for all non-singletone $B\in\mathcal{B}$,

\item[(iii)] $\lambda(B)\rightarrow+\infty$ as $B\rightarrow\{b\}$ for any
$b\in X.$
\end{description}

\bigskip The class of functions $C(B)$ satisfying these conditions is reach
enough, e.g. one can choose
\[
C(B)=(1/m(B))^{\alpha}-(1/m(B^{\prime}))^{\alpha}%
\]
for any two closest neighboring balls $B\subset B^{\prime}$. In this case
\[
\lambda(B)=(1/m(B))^{\alpha}.
\]
Let $\mathcal{D}$ be the set of all locally constant functions having compact
support. The set $\mathcal{D}$ belongs to the Banach spaces $C_{\infty}(X)$
and $L^{p}(X,m),$ $1\leq p<\infty,$ and is a dence subset there.

\emph{The} \emph{homogeneous} \emph{h}$\emph{ierarchical}$ \emph{Laplacian}
$L$ is defined (pointwise) as
\begin{equation}
Lf(x):=\sum\limits_{B\in\mathcal{B}(x)}C(B)\left(  f(x)-\frac{1}{m(B)}%
\int\limits_{B}fdm\right)  \text{.} \label{hlaplacian}%
\end{equation}
The operator $L:\mathcal{D}\rightarrow L^{2}(X,m)$ is symmetric and admits a
complete system of eigenfunctions%
\begin{equation}
f_{B}=\frac{\mathbf{1}_{B}}{m(B)}-\frac{\mathbf{1}_{B^{\prime}}}{m(B^{\prime
})}, \label{eigenfunction}%
\end{equation}
where the couple $B\subset B^{\prime}$ runs over all nearest neighboring balls
having positive measure. The eigenvalue corresponding to $f_{B}$ is
$\lambda(B^{\prime})$ defined above at condition (ii),
\begin{equation}
Lf_{B}(x)=\lambda(B^{\prime})f_{B}(x). \label{eigenvalue}%
\end{equation}
Since the system $\{f_{B}\}$ of eigenfunctions is complete, we conclude that
$L:\mathcal{D}\rightarrow L^{2}(X,m)$ is essentially self-adjoint operator.

\emph{The intrinsic ultrametric} $d_{\ast}(x,y)$ is defined as follows
\begin{equation}
d_{\ast}(x,y):=\left\{
\begin{array}
[c]{ccc}%
0 & \text{when} & x=y\\
1/\lambda(x\curlywedge y) & \text{when} & x\neq y
\end{array}
\right.  , \label{intrinsic ultrametric}%
\end{equation}
where $x\curlywedge y$ is the minimal ball containing both $x$ and $y$. In
particular, for any ball $B$ we have%
\begin{equation}
\lambda(B)=\frac{1}{\mathrm{diam}_{\ast}(B)}. \label{intrinsic diameter}%
\end{equation}
\emph{The spectral function} $\tau\rightarrow N(\tau),$ see equation
(\ref{Jump-kernel}), is defined as a left-continuous step-function having
jumps at the points $\lambda(B)$, and taking values%
\[
N(\lambda(B))=1/m(B).
\]
$\emph{The}$ \emph{volume function }$V(r)$ is defined by setting $V(r)=m(B)$
where the ball $B$ has $d_{\ast}$-radius $r$. It is easy to see that
\begin{equation}
N(\tau)=1/V(1/\tau). \label{Spectral function}%
\end{equation}
The Markovian semigroup $P_{t}=e^{-tL},t>0,$ admits a continuous density
$p(t,x,y)$ w.r.t. $m$, we call it \emph{the heat kernel.} The function
$p(t,x,y)$ can be represented in the form (\ref{d*-jump kernel}).

For $\lambda>0$ the Markovian resolvent $G_{\lambda}=(\lambda+L)^{-1}$ admits
a continuous strictly positive integral kernel $g(\lambda,x,y)$ w. r.t. the
measure $m$. The operator $G_{\lambda}$\ is well defined for $\lambda=0$ (i.e.
the Markovian semigroup $(P_{t})_{t>0}$ is transient) if and only if for some
(equivalently, for all) $x\in X$ the function $\tau\rightarrow1/V(\tau)$ is
integrable at $\infty$. The integral kernel $g(x,y):=g(0,x,y)$, called also
the Green function, is of the form%
\begin{equation}
g(x,y)=%
%TCIMACRO{\dint \limits_{\emph{r}}^{+\infty}}%
%BeginExpansion
{\displaystyle\int\limits_{\emph{r}}^{+\infty}}
%EndExpansion
\frac{d\tau}{V(\tau)},\text{ }r=\emph{d}_{\ast}(x,y). \label{Green function}%
\end{equation}
Under certain Tauberian conditions it takes the form
\begin{equation}
g(x,y)\asymp\frac{r}{V(r)},\text{ }r=\emph{d}_{\ast}(x,y).\footnote{The sign
$\asymp$ means that the ratio of the left- and right hand sides is bounded
from below and above by positive constants.}
\label{Asymptotics_Green function}%
\end{equation}

\subsection{Subordination}

Let $\Phi:\mathbb{R}_{+}\rightarrow\mathbb{R}_{+}$ be an increasing
homeomorphism. For any two nearest neighbouring balls $B\subset B^{\prime}$ we
define
\begin{equation}
C(B)=\Phi\left(  1/m(B)\right)  -\Phi\left(  1/m(B^{\prime})\right)  .
\label{An example}%
\end{equation}
Then the following properties hold:

\begin{description}
\item[(i)] $\lambda(B)=\Phi\left(  1/m(B)\right)  $. In particular, the
corresponding hierarchical Laplacian, denote it $L_{\Phi}$, and the
hierarchical Laplacian $L_{Id}$ are related by the equation $L_{\Phi}%
=\Phi(L_{Id})$.

\item[(ii)] \ $d_{\ast}(x,y)=1/\Phi\left(  1/m(x\curlywedge y)\right)  $.

\item[(iii)] $V(r)\leq1/\Phi^{-1}(1/r).$ Moreover, $V(r)\asymp1/\Phi
^{-1}(1/r)$ whenever both $\Phi$ and $\Phi^{-1}$ are doubling and
$m(B^{\prime})\leq cm(B)$ for some $c>0$ and all neighboring balls $B\subset
B^{\prime}$. In particular, in this case
\[
p_{\Phi}(t,x,y)\asymp t\cdot\min\left\{  \frac{1}{t}\Phi^{-1}\left(  \frac
{1}{t}\right)  ,\frac{1}{m(x\curlywedge y)}\Phi\left(  \frac{1}{m(x\curlywedge
y)}\right)  \right\}  .
\]

\end{description}

\begin{remark}
In general, by Bochner's theorem, for any Markov generator $\mathbb{%
%TCIMACRO{\tciLaplace}%
%BeginExpansion
\mathcal{L}%
%EndExpansion
}$, the operator $-\Phi(-\mathbb{%
%TCIMACRO{\tciLaplace}%
%BeginExpansion
\mathcal{L}%
%EndExpansion
})$ is a Markov generator again provided $\Phi$ is a Bernstein function. It is
known that $\Phi(\lambda)=\lambda^{\alpha}$ is a Bernstein function if and
only if $0<\alpha\leq1$. Thus, for a general Markov generator $\mathbb{%
%TCIMACRO{\tciLaplace}%
%BeginExpansion
\mathcal{L}%
%EndExpansion
}$, the power $-(-\mathbb{%
%TCIMACRO{\tciLaplace}%
%BeginExpansion
\mathcal{L}%
%EndExpansion
})^{\alpha}$ is garanteed to be a Markov generator again only for
$0<\alpha\leq1$. For example, for the classical Laplace operator $\Delta$ in
$\mathbb{R}^{n}$, the power $-(-\Delta)^{\alpha}$ with $\alpha>1$ is not a
Markov generator. This is in striking contrast to the fact that the powers
$L^{\alpha}$ of any hierarchical Laplacian $L$ (remind that $-L$ is always a
Markov generator) are hierarchical Laplacians for all $\alpha>0$, see
\cite[Theorem 3.1]{BGPW}.
\end{remark}

\subsection{Multipliers}

As a special case of the general construction consider $X=\mathbb{Q}_{p}$, the
ring of $p$-adic numbers equipped with its standard ultrametric $\mathrm{d}%
(x,y)=\left\Vert x-y\right\Vert _{p}$. Notice that the ultrametric spaces
$(\mathbb{Q}_{p},\mathrm{d})$ and $(\mathbb{[}0,\infty\mathbb{)},d)$ with
non-eucledian$\ d,$ as explained in the introduction, are isometrically
isomorphic (the isometry can be established via identification of their trees
of metric balls).

Let $m$ be the normed Haar measure on the Abelian group $\mathbb{Q}_{p}$ and
$\widehat{f}$ the Fourier transform of the function $f\in L^{2}(\mathbb{Q}%
_{p},m)$. It is known, see \cite{Taibleson75}, \cite{Vladimirov94},
\cite{Kochubey2004}, that $\mathcal{F}:\mathcal{D}\rightarrow\mathcal{D}$ is a bijection.

Let $\Phi:\mathbb{R}_{+}\rightarrow\mathbb{R}_{+}$ be an increasing
homeomorphism. The self-adjoint operator $\Phi(\mathfrak{D)}$ we define as
multiplier, that is,
\[
\widehat{\Phi(\mathfrak{D)}f}(\xi)=\Phi(\left\Vert \xi\right\Vert
_{p})\widehat{f}(\xi),\text{ \ }\xi\in\mathbb{Q}_{p}.
\]
By \cite[Theorem 3.1]{BGPW}, $\Phi(\mathfrak{D)}$ is a homogeneous
hierarchical Laplacian. The eigenvalues $\lambda(B)$\ of the operator
$\Phi(\mathfrak{D)}$ are numbers
\begin{equation}
\lambda(B)=\Phi\left(  \frac{p}{m(B)}\right)  =\Phi\left(  \frac
{p}{\mathrm{diam}(B)}\right)  . \label{Lambda-Phi eigenvalue}%
\end{equation}
\ Let $p_{\Phi}(t,x,y)$ be the heat kernel associated with the operator
$\Phi(\mathfrak{D}).$ Assuming that both $\Phi$ and $\Phi^{-1}$ are doubling
we get the following relationship
\begin{equation}
p_{\Phi}(t,x,y)\asymp t\cdot\min\left\{  \frac{1}{t}\Phi^{-1}\left(  \frac
{1}{t}\right)  ,\frac{1}{\left\Vert x-y\right\Vert _{p}}\Phi\left(  \frac
{1}{\left\Vert x-y\right\Vert _{p}}\right)  \right\}  . \label{HK-bounds II}%
\end{equation}
The Taibleson-Vladimirov operator $\mathfrak{D}^{\alpha}$ introduced in
\cite{Taibleson75} and \cite{Vladimirov94} is the multiplier corresponding to
the function $\Phi(\tau)=\tau^{\alpha}$ . On the set $\mathcal{D}$ it can be
represented in the form%
\begin{equation}
\mathfrak{D}^{\alpha}\psi(x)=-\frac{1}{\Gamma_{p}(-\alpha)}\int_{\mathbb{Q}%
_{p}}\frac{\psi(x)-\psi(y)}{\left\Vert x-y\right\Vert _{p}^{1+\alpha}}dm(y),
\label{alpha_jump_kernel}%
\end{equation}
where $\Gamma_{p}(z)=(1-p^{z-1})(1-p^{-z})^{-1}$ is the $p$-adic
Gamma-function \cite[Sec.VIII.2, equation (2.17) ]{Vladimirov94}. The function
$z\rightarrow\Gamma_{p}(z)$ is meromorphic in the complex plane $%
%TCIMACRO{\U{2102} }%
%BeginExpansion
\mathbb{C}
%EndExpansion
$ and satisfies the functional equation $\Gamma_{p}(z)\Gamma_{p}(1-z)=1.$

By what we said above, the heat kernel $p_{\alpha}(t,x,y),$ the transition
density of the Markovian semigroup $(e^{-t\mathfrak{D}^{\alpha}})_{t>0}$, can
be uniformly estimated as follows
\begin{equation}
p_{\alpha}(t,x,y)\asymp\frac{t}{(t^{1/\alpha}+\left\Vert x-y\right\Vert
_{p})^{1+\alpha}},
\end{equation}
In particular, the Markov semigroup $(e^{-t\mathfrak{D}^{\alpha}})_{t>0}$ is
transient if and only if $\alpha<1$ (equivalently, the spectral dimension
$D=2/\alpha>2$). In the transient case the Green function $g_{\alpha}(x,y)$
can be computed explicitly%
\begin{equation}
g_{\alpha}(x,y)=\frac{1}{\Gamma_{p}(\alpha)}\frac{1}{\left\Vert x-y\right\Vert
_{p}^{1-\alpha}}.
\end{equation}
For all facts listed above we refer the reader to \cite{BGP}, \cite{BGPW} and
\cite{BendikovKrupski}.

\section{Schr\"{o}dinger-type operators}

Let $(X,d,m)$ be a homogeneous ultrametric measure space and $L$ a homogeneous
hierarchical Laplacian on it. In this section we embark on the study of
Schr\"{o}dinger-type operators acting in $L^{2}(X,m)$. These are operators of
the form%
\begin{equation}
Hf(x)=Lf(x)+V(x)f(x). \label{operator H}%
\end{equation}
The real-valued measurable function $V$ on $X$ is called a potential. Our goal
in this section is to find conditions on the potential $V$ such that one can
associate with equation (\ref{operator H}) a self-adjoint operator $H$.

\subsection{The symbol of the operator $L$}

Identifying $(X,d)$ with a locally compact Abelian group, say $X=\mathbb{Q}%
_{a}$, we can regard $-L$ as an isotropic L\'{e}vy generator. By
(\ref{Spectrum}), the operator $L$ on $\mathcal{D}$ takes the form%
\begin{equation}
Lf(x)=%
%TCIMACRO{\dint \limits_{X}}%
%BeginExpansion
{\displaystyle\int\limits_{X}}
%EndExpansion
(f(x)-f(y))J(x-y)dm(y),\text{ } \label{Levy generator}%
\end{equation}
or equivalently, in terms of the Fourier transform,
\[
\widehat{Lf}(\theta)=\widehat{L}(\theta)\cdot\widehat{f}(\theta),\text{
}\theta\in\widehat{X},
\]
where $\widehat{X}$ is the dual Abelian group which can be identified with
$\mathbb{Q}_{\widehat{a}}$ for a certain sequence $\widehat{a}$, for instance
$\widehat{\mathbb{Q}_{p}}$ can be identified with $\mathbb{Q}_{p}$, and
\begin{equation}
\widehat{L}(\theta)=%
%TCIMACRO{\dint \limits_{X}}%
%BeginExpansion
{\displaystyle\int\limits_{X}}
%EndExpansion
[1-\mathfrak{\operatorname{Re}}\left\langle h,\theta\right\rangle ]J(h)dm(h).
\label{Levy symbol}%
\end{equation}
The function $\widehat{L}(\theta)\geq0$, the symbol of symmetric L\'{e}vy
generator $-L$, is a continuous \emph{negative definite function
}\cite{BergForst}. By \cite[Proposition 7.15]{BergForst}, the function
$\sqrt{\widehat{L}(\theta)}$ is subadditive. By the subordination property
\cite[Theorem 3.1]{BGPW}, the function $\widehat{L}(\theta)^{2}$ is the symbol
of symmetric L\'{e}vy generator $-L^{2}$, so the function $\widehat{L}%
(\theta)=\sqrt{\widehat{L}(\theta)^{2}}$ is subadditive as well, i.e. it
satisfies the triangle inequality
\begin{equation}
\widehat{L}(\theta_{1}+\theta_{2})\leq\widehat{L}(\theta_{1})+\widehat
{L}(\theta_{2}). \label{subadditivity}%
\end{equation}
Since $-L$ is an isotropic L\'{e}vy generator \cite[Sec. 5.2 ]{BGPW} , more
strong property holds true

\begin{theorem}
\label{ultrametric property}The function $\widehat{L}(\theta)$ satisfies the
ultrametric inequality%
\begin{equation}
\widehat{L}(\theta_{1}+\theta_{2})\leq\max\{\widehat{L}(\theta_{1}%
),\widehat{L}(\theta_{2})\}. \label{UM_Symbol}%
\end{equation}

\end{theorem}

\begin{proof}
In order to simplify notation we assume that $X=\mathbb{Q}_{p}$, the ring of
$p$-adic numbers. Let $B\subset B^{\prime}$ be two nearest neighboring balls
centred at the neutral element. Notice that both $B$ and $B^{\prime}$ are
compact subgroups of the group $\mathbb{Q}_{p}$, say $B=$ $p^{-k}%
\mathbb{Z}_{p}$ and $B^{\prime}=$ $p^{-k-1}\mathbb{Z}_{p}$.

Applying the Fourier transform to the both sides of equation (\ref{eigenvalue}%
) we get%
\begin{equation}
\widehat{L}(\theta)\widehat{f_{B}}(\theta)=\lambda(B^{\prime})\widehat{f_{B}%
}(\theta). \label{f-eigenvalue}%
\end{equation}
The measure $\omega_{B}=(\mathbf{1}_{B}m)/m(B)$ is the normalized Haar measure
of the compact subgroup $B$, similarly $\omega_{B^{\prime}}$. Since for any
locally compact Abelian group, the Fourier transform of the normalized Haar
measure of any compact subgroup $A$ is the indicator of its annihilator group
$A^{\bot}$, and in our particular case $B^{\bot}=$ $p^{k}\mathbb{Z}_{p}$ and
$(B^{\prime})^{\bot}=$ $p^{k+1}\mathbb{Z}_{p}$, we obtain
\begin{equation}
\widehat{f_{B}}(\theta)=\mathbf{1}_{B^{\bot}}(\theta)-\mathbf{1}_{(B^{\prime
})^{\bot}}(\theta)=\mathbf{1}_{\partial B^{\bot}}(\theta),
\label{f-eigenfunction}%
\end{equation}
where $\partial B^{\bot}$ is the sphere $B^{\bot}\setminus(B^{\prime})^{\bot}$.

Equations (\ref{f-eigenfunction}) and (\ref{eigenvalue}) imply that the
function $\widehat{L}(\theta)$ takes constant value $\lambda(B^{\prime})$ on
the sphere $\partial B^{\bot}$, i.e. $\widehat{L}(\theta)=\psi(\left\Vert
\theta\right\Vert _{p})$ for some function $\psi(\tau)$ such that $\psi(0)=0$
and $\psi(+\infty)=+\infty.$ Since $C\subset D$ implies $\lambda
(C)>\lambda(D)$, the function $\psi(\tau)$ can be chosen to be continuous and
increasing, so $\widehat{L}(\theta)=\psi(\left\Vert \theta\right\Vert _{p})$
satisfies the ultrametric inequality (\ref{UM_Symbol}) as claimed.
\end{proof}

\subsection{Locally bounded potentials}

If we assume that the potential $V$ is a locally bounded function then
\[
(Hu)(x):=(Lu)(x)+V(x)u(x)
\]
is a well defined symmetric operator $H:$ $\mathcal{D}\rightarrow L^{2}(X,m)$.
For the proof of the following theorem we refer to the paper \cite[Theorem
3.1]{BGM}

\begin{theorem}
\label{Schroedinger spectrum}Assume that $V$ is a locally bounded function, then

1. The operator $H$ is essentially self-adjoint.

2. If $V(x)\rightarrow+\infty$ as $x\rightarrow\varpi$, then the self-adjoint
operator $H$ has a compact resolvent. (Thus, its spectrum is discrete).

3. If $V(x)\rightarrow0$ as $x\rightarrow\varpi$, then the essential spectrum
of $H$ coincides with the spectrum of $L$. (Thus, the spectrum of $H$ is pure
point and the negative part of the spectrum consists of isolated eigenvalues
of finite multiplicity).

\begin{remark}
For the classical Schr\"{o}dinger operator $H=-\Delta+V$ defined on the set of
test functions $C_{0}^{\infty}$ the statement similar to the statement 1 of
Theorem \ref{Schroedinger spectrum} is known as the Sears's theorem: $H$ is
essentially self-adjoint if the potential $V$ admits a low bound%
\[
V(x)\geq Q(\left\vert x\right\vert ),
\]
where $0\leq Q(r)\in C(\mathbb{R}_{+})$ a non-decreasing function such that%
\[
\int_{0}^{\infty}Q(r)^{-1/2}dr=\infty,
\]
and it may fail to be essentially self-adjoint otherwise, see \cite[Chapter
II, Theorem 1.1 and Example 1.1]{BeresinShubin}.
\end{remark}
\end{theorem}

\subsection{Potentials with local singularities}

If we are \ interested in potentials with local singularities, such as e.g.
$V(x)=\left\Vert x\right\Vert _{p}^{-\beta}$, then certain local conditions on
the potential are necessary in order to prove that the quadratic form%
\begin{equation}
Q(u,u)=Q_{L}(u,u)+Q_{V}(u,u) \label{quadratic form}%
\end{equation}
defined on the set%
\[
dom(Q):=dom(Q_{L})\cap dom(Q_{V})
\]
is a densly defined closed and bounded below quadratic form and whence it is
associated to a bounded below self-adjoint operator $H$. That means precisely
that there exists a constant $c>0$ and a self-adjoint operator $H$ such that
the form $Q^{\prime}(u,u):=Q(u,u)+c(u,u)$ (resp. the operator $H+c\mathrm{I}$)
is non-negative definite and that
\begin{equation}
Q^{\prime}(u,u)=((H+c\mathrm{I})^{1/2}u,(H+c\mathrm{I})^{1/2}u)
\label{min-extension}%
\end{equation}
for all $u\in dom(Q)$.

It is customary to write $H=L+V$, but it must be remembered that this is a
quadratic form sum and not an operator sum as in the previous subsection.

\begin{theorem}
\label{L^1_loc theorem}If $\ 0\leq V\in L_{loc}^{1}(X,m)$, then the quadratic
form (\ref{quadratic form}) is a regular Dirichlet form \cite{Fukushima}. In
particular, it is the form of a non-negative self-adjoint operator $H$,
\[
Q(u,u)=(H^{1/2}u,H^{1/2}u)
\]
and the set $\mathcal{D}$ is a core for $Q$.
\end{theorem}

\begin{proof}
The set $\mathcal{D}$ beongs to both $dom(Q_{L})$ and $dom(Q_{V})$ hence $Q$
is densly defined. Set $V_{\tau}=V\wedge\tau$ and define on the set
$dom(Q_{L})$ the form
\[
Q^{\tau}(u,u)=Q_{L}(u,u)+Q_{V_{\tau}}(u,u).
\]
Since $V_{\tau}$ is bounded the form $Q^{\tau}$ is closed. In particular, the
function $u\rightarrow Q^{\tau}(u,u)$ is lower semicontinuous. Clearly
$Q(u,u)=\sup\{Q^{\tau}(u,u):\tau>0\}$. It follows that the function
$u\rightarrow Q(u,u)$ is lower semicontinuous. Hence by \cite[Theorem
4.4.2]{Davies} the form $Q$ is closed, and thus it is the form of a
non-negative definite self-adjoint operator $H$. Clearly the form $Q$ is
Markovian hence it is a Dirichlet form.

Let us show that $\mathcal{D}$ is a core for $Q$, i.e. $Q$ is a regular
Dirichlet form.

\textbf{Step 1} For $u\in dom(Q)$ we set $u_{n}=((-n)\vee u)\wedge n$, then
$u_{n}\in dom(Q)$ and $Q(u-u_{n},u-u_{n})\rightarrow0$, see \cite[Theorem
1.4.2]{Fukushima}. Therefore the set of bounded functions in $dom(Q)$ is a
core for $Q.$

\textbf{Step 2} Let $B$ be a ball centred at the neutral element. Let $u\in
dom(Q)$ be bounded and $u_{B}=1_{B}\cdot u$. The function $1_{B}$ is in
$dom(Q)$ (and even in $dom(L)$), whence applying \cite[Theorem 1.4.2]%
{Fukushima} we get: $u_{B}\in dom(Q)$ and%
\[
\sqrt{Q(u_{B},u_{B})}\leq\sqrt{Q(u,u)}+\left\Vert u\right\Vert _{\infty}%
\cdot\sqrt{Q(1_{B},1_{B})}.
\]
It is straightforward to show that%
\[
m(B)\lambda(B^{\prime})\leq Q(1_{B},1_{B})\leq2m(B)\lambda(B^{\prime}),
\]
where $B^{\prime}$ is the closest neighboring ball containing $B$, and
$\lambda(B^{\prime})$ is the eigenvalue of $L$ corresponding to the ball
$B^{\prime}$. Thus, if we assume that%
\begin{equation}
\lim_{B\nearrow X}m(B)\lambda(B^{\prime})=0\text{ \ \ }%
\ \label{intermediate condition}%
\end{equation}
(as it happens in the case of the operator $L=\mathfrak{D}^{\alpha},$
$\ \alpha>1$) then the following contraction property holds%
\begin{equation}
\limsup_{B\nearrow X}Q(u_{B},u_{B})\leq Q(u,u). \label{almost contraction}%
\end{equation}
Let $(R_{\lambda})_{\lambda>0}$ be the Markov resolvent corresponding to $Q$.
Let $Q_{1}(s,t):=Q(s,t)+(s,t)$. Then for any $v\in L^{2}(X,m),$
\[
Q_{1}(u_{B},R_{1}v)=(u_{B},v)\rightarrow(u,v)=Q_{1}(u,R_{1}v)
\]
Since $R_{1}(L^{2}(X,m))$ is dense in $dom(Q)$ with respect to the metric
$Q_{1}$, the sequence $u_{B}$ weakly converges to $u$ with respect to $Q_{1}%
$:
\begin{equation}
Q_{1}(u_{B},w)\rightarrow Q_{1}(u,w),\text{ \ }\forall w\in dom(Q).
\label{weak convergence}%
\end{equation}
Using equations (\ref{almost contraction}) and (\ref{weak convergence}) we
obtain:
\begin{align*}
\limsup_{B\nearrow X}Q_{1}(u-u_{B},u-u_{B})  &  =\limsup_{B\nearrow X}\left(
Q_{1}(u,u)-2Q_{1}(u_{B},u)+Q_{1}(u_{B},u_{B})\right) \\
&  =Q_{1}(u,u)-2\lim_{B\nearrow X}Q_{1}(u_{B},u)+\limsup_{B\nearrow X}%
Q_{1}(u_{B},u_{B})\\
&  \leq Q_{1}(u,u)-2Q_{1}(u,u)+Q_{1}(u,u)=0.
\end{align*}
Thus, if condition (\ref{intermediate condition}) holds, the set of bounded
functions with compact support in $dom(Q)$ is a core for $Q$ as desired.

\textbf{Step 3} In general to prove contraction property
(\ref{almost contraction}) we proceed as follows. Any ball $B$ centred at the
neutral element is the compact subgroup of $X$. Since the Fourier transform of
the normalized Haar measure of a compact subgroup is the indicator of its
annihilator group, we obtain
\begin{align*}
Q_{L}(u_{B},u_{B})  &  =\int_{\widehat{X}}\widehat{L}(\theta)\left\vert
\widehat{u_{B}}(\theta)\right\vert ^{2}d\widehat{m}(\theta)\\
&  =\int_{\widehat{X}}\widehat{L}(\theta)\left\vert \widehat{u}\ast\widehat
{m}_{B^{\perp}}(\theta)\right\vert ^{2}d\widehat{m}(\theta),
\end{align*}
where $\widehat{L}(\theta)$ is the sumbol of the multiplier $L$, $B^{\perp}$
is the annihilator group of the compact subgroup $B\subset X$ and $\widehat
{m}_{B^{\perp}}$ is the normed Haar measure of $B^{\perp}.$ The function
$\widehat{L}(\theta)$ is an increasing function of the $p$-adic norm
$\left\Vert \theta\right\Vert _{p}$ whence it satisfies the ultrametric
inequality. Having this in mind and using the inequality%
\[
\left\vert \widehat{u}\ast\widehat{m}_{B^{\perp}}\right\vert ^{2}%
\leq\left\vert \widehat{u}\right\vert ^{2}\ast\widehat{m}_{B^{\perp}}%
\]
we get
\begin{align*}
Q_{L}(u_{B},u_{B})  &  \leq\int_{\widehat{X}}\widehat{L}(\theta)\left(
\left\vert \widehat{u}\right\vert ^{2}\ast\widehat{m}_{B^{\perp}}\right)
(\theta)d\widehat{m}(\theta)\\
&  =\int_{\widehat{X}}\widehat{L}(\theta)\left(
%TCIMACRO{\dint _{B^{\perp}}}%
%BeginExpansion
{\displaystyle\int_{B^{\perp}}}
%EndExpansion
\left\vert \widehat{u}(\theta+\zeta)\right\vert ^{2}d\widehat{m}_{B^{\perp}%
}(\zeta)\right)  d\widehat{m}(\theta)\\
&  =\int_{B^{\perp}}\left(  \int_{\widehat{X}}\widehat{L}(\theta
+\zeta)\left\vert \widehat{u}(\theta)\right\vert ^{2}d\widehat{m}%
(\theta)\right)  d\widehat{m}_{B^{\perp}}(\zeta).
\end{align*}
Whence, using the ultrametric inequality (\ref{UM_Symbol}), we obtain%
\begin{align*}
Q_{L}(u_{B},u_{B})  &  \leq\int_{B^{\perp}}\left(  \int_{\widehat{X}}%
\max\left\{  \widehat{L}(\theta),\widehat{L}(\zeta)\right\}  \left\vert
\widehat{u}(\theta)\right\vert ^{2}d\widehat{m}(\theta)\right)  d\widehat
{m}_{B^{\perp}}(\zeta)\\
&  \leq\int_{B^{\perp}}\left(  \int_{\widehat{X}}\left(  \widehat{L}%
(\theta)+\widehat{L}(\zeta)\right)  \left\vert \widehat{u}(\theta)\right\vert
^{2}d\widehat{m}(\theta)\right)  d\widehat{m}_{B^{\perp}}(\zeta)\\
&  =Q_{L}(u,u)+\left(  \int_{B^{\perp}}\widehat{L}(\zeta)d\widehat
{m}_{B^{\perp}}(\zeta)\right)  (u,u).
\end{align*}
When $B\nearrow X$ the measure $\widehat{m}_{B^{\perp}}$ converges weakly to
the Dirac measure concentrated at the neutral element, whence we finally
obtain the following inequality%
\begin{equation}
\limsup_{B\nearrow X}Q_{L}(u_{B},u_{B})\leq Q_{L}(u,u).
\label{almost contraction'}%
\end{equation}
Evidently inequality (\ref{almost contraction'}) implies inequality
(\ref{almost contraction}) as desired.

\textbf{Step 4} Now let $u\in dom(Q)$ be bounded and of compact support. Let
$B$ be a ball centred at the neutral element of $X$\ (a compact subgroup of
$X$) and $m_{B}$ be its normed Haar measure. We set $u^{B}=u\ast m_{B}$. The
function $u^{B}$ is locally constant and has a compact support, hence belongs
to $dom(Q)$. We have $\widehat{u^{B}}=\widehat{u}\cdot1_{B^{\bot}}$ whence%
\[
\left\Vert u-u^{B}\right\Vert _{2}^{2}=%
%TCIMACRO{\dint \limits_{(B^{\bot})^{c}}}%
%BeginExpansion
{\displaystyle\int\limits_{(B^{\bot})^{c}}}
%EndExpansion
\left\vert \widehat{u}(\theta)\right\vert ^{2}d\widehat{m}(\theta)
\]
which converges to zero as $B$ converges to the trivial subgroup $\{e\}$. A
similar argument establishes that%
\[
\lim_{B\rightarrow\{e\}}Q_{L}(u-u^{B},u-u^{B})=\lim_{B\rightarrow\{e\}}%
%TCIMACRO{\dint \limits_{(B^{\bot})^{c}}}%
%BeginExpansion
{\displaystyle\int\limits_{(B^{\bot})^{c}}}
%EndExpansion
\widehat{L}(\theta)\left\vert \widehat{u}(\theta)\right\vert ^{2}d\widehat
{m}(\theta)=0.
\]
There exists a compact set $K$ which contains the support of every $u-u^{B}$
when the diameter of $B$ is less or equal one. Given $\varepsilon>0$ there
exists a decomposition $V|_{K}=V_{1}+V_{2}$ such that $\left\Vert
V_{1}\right\Vert _{1}<\varepsilon$ and $V_{2}\in L^{\infty}(X,m)$. We then
have%
\begin{align*}
Q_{V}(u-u^{B},u-u^{B})  &  =%
%TCIMACRO{\dint \limits_{K}}%
%BeginExpansion
{\displaystyle\int\limits_{K}}
%EndExpansion
V\left\vert u-u^{B}\right\vert ^{2}dm\\
&  =%
%TCIMACRO{\dint \limits_{K}}%
%BeginExpansion
{\displaystyle\int\limits_{K}}
%EndExpansion
V_{1}\left\vert u-u^{B}\right\vert ^{2}dm+%
%TCIMACRO{\dint \limits_{K}}%
%BeginExpansion
{\displaystyle\int\limits_{K}}
%EndExpansion
V_{2}\left\vert u-u^{B}\right\vert ^{2}dm\\
&  \leq4\varepsilon\left\Vert u\right\Vert _{\infty}^{2}+\left\Vert
V_{2}\right\Vert _{\infty}\left\Vert u-u^{B}\right\Vert _{2}^{2}.
\end{align*}
Therefore%
\[
\limsup_{B\rightarrow\{e\}}Q_{V}(u-u^{B},u-u^{B})\leq4\varepsilon\left\Vert
u\right\Vert _{\infty}^{2}%
\]
for all $\varepsilon>0$. In other words%
\[
\lim_{B\rightarrow\{e\}}Q_{1}(u-u^{B},u-u^{B})=0
\]
and thus $\mathcal{D}$ is indeed a core for $Q=Q_{L}+Q_{V}$.
\end{proof}

\begin{remark}
It is clear that the above theorem can be extended to $V$ which are bounded
below and in $L_{loc}^{1}(X,m)$ by simply adding a large enough positive
constant. If, however, we are interested in $V$ with negative local
singularities, then stronger local conditions on $V$ are necessary in order to
be able to prove that the form $Q$ is closed.
\end{remark}

\begin{definition}
Let $p\geq1$ be fixed. We say that a potential $V$ lies in $L^{p}+L^{\infty}$
if one can write $V=V^{\prime}+V^{\prime\prime}$where $V^{\prime}\in
L^{p}(X,m)$ and $V^{\prime\prime}\in L^{\infty}(X,m)$ . This decomposition is
not unique, and, if it is possible at all, then one can arrange for
$\left\Vert V^{\prime}\right\Vert _{p}$ to be as small as one chooses.
\end{definition}

\begin{theorem}
\label{L^p-L^infty theorem}Let $L=\mathfrak{D}^{\gamma}$, and let
$Q=Q_{L}+Q_{V}$ be quadratic form (\ref{quadratic form}) where $V\in
L^{p}+L^{\infty}$ for some $p>1/\gamma$. Then:

1. $Q$ is a densly defined closed and bounded below form whence it is
associated with a bounded below self-adjoint operator $H$.

2. If $2\leq1/\gamma<p$ then $dom(H)=dom(\mathfrak{D}^{\gamma})$. The same is
true if $1/\gamma<2$ and $p=2$ .
\end{theorem}

\begin{proof}
The set $\mathcal{D}$ is in both $dom(Q_{L})$ and $dom(Q_{V})$ whence $Q$ is
densly defined. Given $\varepsilon>0$ we may write $\left\vert V\right\vert
=W+W^{\prime}$ where $\left\Vert W\right\Vert _{p}<\varepsilon$ and
$W^{\prime}\in L^{\infty}(X,m)$. We claim that if $\varepsilon>0$ is
sufficiently small, then%
\begin{equation}
\left\Vert W^{1/2}u\right\Vert _{2}^{2}\leq\frac{1}{2}Q_{L}(u,u)+c_{0}%
\left\Vert u\right\Vert _{2}^{2} \label{claim}%
\end{equation}
for some constant $c_{0}>0$ and all $u\in dom(Q_{L})$.

Clearly inequality \ref{claim} yield that%
\begin{align*}%
%TCIMACRO{\dint }%
%BeginExpansion
{\displaystyle\int}
%EndExpansion
\left\vert V\right\vert \left\vert u\right\vert ^{2}dm  &  \leq\left\Vert
W^{1/2}u\right\Vert _{2}^{2}+\left\Vert W^{\prime}\right\Vert _{\infty
}\left\Vert u\right\Vert _{2}^{2}\\
&  \leq\frac{1}{2}Q_{L}(u,u)+c_{1}\left\Vert u\right\Vert _{2}^{2}%
\end{align*}
for some constant $c_{1}>0$ and all $u\in dom(Q_{L})$. Thus for $c_{2}>2c_{1}$
we get%
\[
\frac{1}{2}\left\{  Q_{L}(u,u)+c_{2}\left\Vert u\right\Vert _{2}^{2}\right\}
\leq Q(u,u)+c_{2}\left\Vert u\right\Vert _{2}^{2}\leq\frac{3}{2}\left\{
Q_{L}(u,u)+c_{2}\left\Vert u\right\Vert _{2}^{2}\right\}  .
\]
It follows that the quadratic form $u\rightarrow Q(u,u)+c_{2}\left\Vert
u\right\Vert _{2}^{2}$ is non-negative and closed whence it is associated with
a non-negative self-adjoint operator, which is clearly equal to $H+c_{2}I$.

To prove the claim \ref{claim} we need some $L^{p}$-estimates. Recall that the
number $D=2/\gamma$ is called \emph{the spectral dimension }related to the
operator $\mathfrak{D}^{\gamma}$. The estimates $(E.1)$ and $(E.2)$ below are
similar to the classical estimates for the Hamiltonian $-\Delta$ in the
Eucledian space $\mathbb{R}^{D}$, see \cite[Sec. 3.6]{Davies}.

\begin{description}
\item[E1.] If $0<\alpha\leq1/(2\gamma)$ and $2\leq p<2/(1-2\alpha\gamma)$,
then $(\mathfrak{D}^{\gamma}+\mathrm{I})^{-\alpha}$ is a bounded linear
operator from $L^{2}(X,m)$ to $L^{p}(X,m)$. If $\alpha>1/(2\gamma)$, then
$(\mathfrak{D}^{\gamma}+\mathrm{I})^{-\alpha}$ is a bounded linear operator
from $L^{2}(X,m)$ to $L^{\infty}(X,m)$.

\item[E2.] If $0<\alpha\leq1/(2\gamma)$ and $\mathcal{W}\in L^{q}(X,m)$ is a
multiplication linear operator, then $\mathcal{A}:\mathcal{=W\cdot
}(\mathfrak{D}^{\gamma}+\lambda\mathrm{I})^{-\alpha}$ is a bounded linear
operator on $L^{2}(X,m)$ provided $1/(\alpha\gamma)<q\leq\infty$. Moreover,
there exists a constant $c>0$ such that $\left\Vert \mathcal{A}\right\Vert
_{L^{2}\rightarrow L^{2}}\leq c\left\Vert \mathcal{W}\right\Vert _{q}$ for all
such $\mathcal{W}$. The same bound holds in the case $\alpha>1/(2\gamma)$ and
$q=2$. In both cases the operator $\mathcal{A}$ is a compact operator on
$L^{2}$. Moreover, $\lim_{\lambda\rightarrow\infty}\left\Vert \mathcal{W\cdot
}(\mathfrak{D}^{\gamma}+\lambda\mathrm{I})^{-\alpha}\right\Vert _{L^{2}%
\rightarrow L^{2}}=0$.
\end{description}

Proof of the statement E1. Assume first that $0<\alpha\leq1/(2\gamma)$. If we
define the function $g(y):=(\left\Vert y\right\Vert _{p}^{\gamma}+1)^{-\alpha
}$ and assume that $1/(\alpha\gamma)<s\leq\infty$ then%
\[
\left\Vert g\right\Vert _{s}^{s}=%
%TCIMACRO{\dint \limits_{\mathbb{Q}_{p}}}%
%BeginExpansion
{\displaystyle\int\limits_{\mathbb{Q}_{p}}}
%EndExpansion
\frac{dm(y)}{(\left\Vert y\right\Vert _{p}^{\gamma}+1)^{\alpha s}}=\left(
1-\frac{1}{p}\right)
%TCIMACRO{\dsum \limits_{\tau=-\infty}^{\infty}}%
%BeginExpansion
{\displaystyle\sum\limits_{\tau=-\infty}^{\infty}}
%EndExpansion
\frac{p^{\tau}}{(p^{\tau\gamma}+1)^{\alpha s}}<\infty.
\]
If $k=\widehat{(\mathfrak{D}^{\gamma}+I)^{-\alpha}f}$ and $f\in L^{2}$, then
$k(y)=g(y)\widehat{f}(y)$. Putting $1/q=1/s+1/2$ we deduce that $1<q\leq2$ and%
\[
\left\Vert k\right\Vert _{q}\leq\left\Vert g\right\Vert _{q}\left\Vert
\widehat{f}\right\Vert _{2}=c_{1}\left\Vert f\right\Vert _{2}.
\]
If $1/p+1/q=1$, then $2\leq p<\infty$ and it follows from the Hausdorff-Young
theorem that%
\[
\left\Vert (\mathfrak{D}^{\gamma}+I)^{-\alpha}f\right\Vert _{p}=\left\Vert
\widehat{k}\right\Vert _{p}\leq\left\Vert k\right\Vert _{q}\leq c_{1}%
\left\Vert f\right\Vert _{2}.
\]
We have $1/p=1-1/q=1/2-1/s$ and $1/(\alpha\gamma)<s\leq\infty$, whence $p$
increases from $2$ to $2/(1-2\alpha\gamma)$ as $s$ decreases from $\infty$ to
$1/(\alpha\gamma)$.

If $\alpha>1/(2\gamma)$, then the function $g$ defined above lies in $L^{2}$
and we deduce that%
\[
\left\Vert k\right\Vert _{1}=\left\Vert g\widehat{f}\right\Vert _{1}%
\leq\left\Vert g\right\Vert _{2}\left\Vert \widehat{f}\right\Vert _{2}%
=c_{2}\left\Vert f\right\Vert _{2}%
\]
whence as above%
\[
\left\Vert (\mathfrak{D}^{\gamma}+I)^{-\alpha}f\right\Vert _{\infty
}=\left\Vert \widehat{k}\right\Vert _{\infty}\leq\left\Vert k\right\Vert
_{1}\leq c_{2}\left\Vert f\right\Vert _{2}%
\]
as desired.

Proof of the statement E2. If $0<\alpha\leq1/(2\gamma)$, then%
\[
\left\Vert \mathcal{W\cdot}(\mathfrak{D}^{\gamma}+I)^{-\alpha}f\right\Vert
_{2}\leq\left\Vert \mathcal{W}\right\Vert _{q}\left\Vert (\mathfrak{D}%
^{\gamma}+I)^{-\alpha}f\right\Vert _{p}%
\]
provided $1/2=1/p+1/q$. The condition $2\leq p<2/(1-2\alpha\gamma)$ is
equivalent to $1/(\alpha\gamma)<q\leq\infty$. We apply the statement E1 to get
the desired conclusion. The case $\alpha>1/(2\gamma)$ is similar,%
\[
\left\Vert \mathcal{W\cdot}(\mathfrak{D}^{\gamma}+I)^{-\alpha}f\right\Vert
_{2}\leq\left\Vert \mathcal{W}\right\Vert _{2}\left\Vert (\mathfrak{D}%
^{\gamma}+I)^{-\alpha}f\right\Vert _{\infty}.
\]
To prove compactness of the operator $\mathcal{A=W\cdot}(\mathfrak{D}^{\gamma
}+I)^{-\alpha}$ we choose a sequence $\mathcal{W}_{n}\in\mathcal{D}$ such that
$\mathcal{W}_{n}\rightarrow\mathcal{W}$ in $L^{q}$. Let $\Phi_{n}$ be a
strictly increasing function such that $\Phi_{n}(\tau)=\tau^{\gamma}$ for
$0\leq\tau\leq n$ and $\Phi_{n}(\tau)\asymp e^{\tau}$ as $\tau\rightarrow
\infty$. If we set $\mathcal{A}_{n}=$ $\mathcal{W}_{n}\mathcal{\cdot}(\Phi
_{n}(\mathfrak{D})+I)^{-\alpha}$ then $\mathcal{A}_{n}\rightarrow\mathcal{A}$
in the operator norm. Since the set of compact operators is closed under norm
limits, it is sufficient to prove that each $\mathcal{A}_{n}$ is a
Hilbert-Schmidt operator. Each operator $\mathcal{A}_{n}$ is unitary
equivalent to the integral operator $\widehat{\mathcal{A}_{n}}:\widehat
{u}\rightarrow\widehat{\mathcal{A}_{n}u}$ which has the kernel%
\[
\widehat{\mathcal{A}_{n}}(\theta,\zeta)=\widehat{\mathcal{W}_{n}}(\theta
-\zeta)(\Phi_{n}(\left\Vert \zeta\right\Vert )+1)^{-\alpha}:=\widehat
{\mathcal{W}_{n}}(\theta-\zeta)\mathcal{G(\zeta)}%
\]
so that the Hilbert-Schmidt norm $\left\Vert \widehat{\mathcal{A}_{n}%
}\right\Vert $of the operator $\widehat{\mathcal{A}_{n}}$ \ is%
\[
\left\Vert \widehat{\mathcal{A}_{n}}\right\Vert =\left\Vert \mathcal{W}%
_{n}\right\Vert _{2}\left\Vert \mathcal{G}\right\Vert _{2}<\infty.
\]
Thus the operator $\mathcal{A=W\cdot}(\mathfrak{D}^{\gamma}+I)^{-\alpha}$ is
endeed a compact operator.

Let us turn to the proof of the claim \ref{claim}. To prove the claim in the
case $0<\gamma\leq1$ and $p>1/\gamma$ we write%
\begin{align*}
\left\Vert W^{1/2}u\right\Vert _{2}^{2}  &  =\left\Vert W^{1/2}\cdot
(\mathfrak{D}^{\gamma}+I)^{-1/2}\cdot(\mathfrak{D}^{\gamma}+I)^{1/2}%
u\right\Vert _{2}^{2}\\
&  \leq\left\Vert W^{1/2}\cdot(\mathfrak{D}^{\gamma}+I)^{-1/2}\right\Vert
_{L^{2}\rightarrow L^{2}}^{2}\left\Vert (\mathfrak{D}^{\gamma}+I)^{1/2}%
u\right\Vert _{2}^{2}\\
&  =\left\Vert W^{1/2}\cdot(\mathfrak{D}^{\gamma}+I)^{-1/2}\right\Vert
_{L^{2}\rightarrow L^{2}}^{2}\left(  Q_{L}(u,u)+\left\Vert u\right\Vert
_{2}^{2}\right) \\
&  \leq c\left\Vert W^{1/2}\right\Vert _{q}^{2}\left(  Q_{L}(u,u)+\left\Vert
u\right\Vert _{2}^{2}\right)  \leq\frac{1}{2}Q_{L}(u,u)+c_{1}\left\Vert
u\right\Vert _{2}^{2}%
\end{align*}
provided $\varepsilon>0$ is chosen small enough and $q=2p>2/\gamma$ as in the
statement E2 with $\alpha=1/2$.

The case $\gamma>1$ is similar: The restriction $p>1/\gamma$ becomes $p\geq1.$
We set $Y=\{\left\vert V\right\vert >\tau\}$ and $W=\left\vert V\right\vert
1_{Y}$. By Markov inequality $m(Y)\leq\tau^{-p}\left\Vert V\right\Vert
_{p}^{p}<\infty$ whence $\left\Vert W\right\Vert _{1}=o(1)$ as $\tau
\rightarrow\infty$. In particular, $W^{1/2}\in L^{2}$ and $\left\Vert
W^{1/2}\right\Vert _{2}=o(1)$ as $\tau\rightarrow\infty$. Applying the second
part of the statement E2 with $\alpha=1/2$ and $q=2$ we come to the conclusion%
\begin{align*}
\left\Vert W^{1/2}u\right\Vert _{2}^{2}  &  \leq c\left\Vert W^{1/2}%
\right\Vert _{2}^{2}\left(  Q_{L}(u,u)+\left\Vert u\right\Vert _{2}^{2}\right)
\\
&  \leq\frac{1}{2}Q_{L}(u,u)+c_{1}\left\Vert u\right\Vert _{2}^{2},
\end{align*}
as desired.

To prove that $dom(H)=dom(\mathfrak{D}^{\gamma})$ we first write $V=V^{\prime
}+V^{\prime\prime}$, where $V^{\prime}\in L^{p}(X,m)$ and $V^{\prime\prime}\in
L^{\infty}(X,m)$. The statement E2 yields that%
\[
\lim_{t\rightarrow\infty}\left\Vert V^{\prime}\cdot(\mathfrak{D}^{\gamma
}+t\mathrm{I})^{-1}\right\Vert _{L^{2}\rightarrow L^{2}}=0.
\]
We also have%
\[
\left\Vert V^{\prime\prime}\cdot(\mathfrak{D}^{\gamma}+t\mathrm{I}%
)^{-1}\right\Vert _{L^{2}\rightarrow L^{2}}\leq\left\Vert V^{\prime\prime
}\right\Vert _{\infty}\left\Vert (\mathfrak{D}^{\gamma}+t\mathrm{I}%
)^{-1}\right\Vert _{L^{2}\rightarrow L^{2}}=t^{-1}\left\Vert V^{\prime\prime
}\right\Vert _{\infty}%
\]
for all $t>0$, so%
\[
\lim_{t\rightarrow\infty}\left\Vert V\cdot(\mathfrak{D}^{\gamma}%
+t\mathrm{I})^{-1}\right\Vert _{L^{2}\rightarrow L^{2}}=0.
\]
For any $1>\delta>0$ small enough we conclude that if $t>0$ is large enough
then%
\[
\left\Vert Vf\right\Vert _{2}\leq\delta\left\Vert \mathfrak{D}^{\gamma
}f\right\Vert _{2}+t\delta\left\Vert f\right\Vert _{2}%
\]
for all $f\in dom(\mathfrak{D}^{\gamma})$. Thus $V$ is a relatively bounded
perturbation of $\mathfrak{D}^{\gamma}$with a relative bound $\delta<1$ whence
$dom(\mathfrak{D}^{\gamma}+V)=dom(\mathfrak{D}^{\gamma})$ by an application of
\cite[Theorem 1.4.2]{Davies}. The proof is now completed
\end{proof}

Next we discuss several results giving information about the negative part of
the spectrum of $H$.

\begin{theorem}
\label{negative spec}Let $L=\mathfrak{D}^{\gamma}$ and let $V\in L^{p}(X,m)$
for some $p>1/\gamma$. Then:

1. The operator $H=L+V$ has essential spectrum equals to the spectrum of $L$.
In particular, if $H$ has any negative spectrum, then it consists of a
sequence of negative eigenvalues of finite multiplicity. If this sequence is
infinite then it converges to zero.

2. Suppose that there exists an open set $U\subset X$ on which $V$ is
negative. If $E_{\lambda}$ is the bottom of the spectrum of the operator
$H_{\lambda}=L+\lambda V$, then $E_{\lambda}\leq0$ for all $\lambda\geq0$ and
$\lim_{\lambda\rightarrow\infty}E_{\lambda}=-\infty$.
\end{theorem}

\begin{proof}
1. By Theorem \ref{L^p-L^infty theorem}, if $c>0$ is large enough then the
operator $H+cI$ is non-negative and%
\begin{equation}
\frac{1}{2}\left\Vert (L+cI)^{1/2}u\right\Vert _{2}\leq\left\Vert
(H+cI)^{1/2}u\right\Vert _{2}\leq\frac{3}{2}\left\Vert (L+cI)^{1/2}%
u\right\Vert _{2} \label{equiv QF}%
\end{equation}
for all $u\in dom(Q_{L})$. Let us define $\Delta:=(L+cI)^{-1}-(H+cI)^{-1}$,
then
\[
\Delta=(L+cI)^{-1}V(H+cI)^{-1}=ABCDE
\]
where $A=(L+cI)^{-1/2}$, $B=(L+cI)^{-1/2}\left\vert V\right\vert ^{1/2}$,
$C=sign(V)B^{\ast},D=(L+cI)^{1/2}(H+cI)^{-1/2}$ and $E=(H+cI)^{-1/2}$. It is
clear that $A$ and $E$ are bounded operators on $L^{2}$, $B^{\ast}$ and $C$
are compact operators on $L^{2}((X,m)$, see the statement E2 in the proof of
Theorem \ref{L^p-L^infty theorem}, and $D$ is a bounded operator on
$L^{2}(X,m)$ by equation (\ref{equiv QF}). Thus, as a product of compact and
bounded operators, the difference of two resolvents $\Delta$ is a compact
operator on $L^{2}$. By perturbation theory of linear operators, $H$ and $L$
have the same essential spectrum, see e.g. \cite{Kato}. Since $Spec_{ess}%
(L)=Spec(L)\subset\lbrack0,\infty\lbrack$, any negative point in the spectrum
of $H$ must be an isolated eigenvalue of finite multiplicity. Any limit of
negative eigenvalues lies in the essential spectrum whence the only possible
limit is zero.

2. We use the first statement to prove that $E_{\lambda}\leq0$ for all
$\lambda\geq0$. Observe that%
\begin{equation}
E_{\lambda}=\inf\{Q_{L}(u,u)+\lambda Q_{V}(u,u):u\in\mathcal{D}\text{ and
}\left\Vert u\right\Vert _{2}=1\} \label{variation formula}%
\end{equation}
because $\mathcal{D}$ is a core for $Q_{L}+\lambda Q_{V}$. Let us choose
$u\in\mathcal{D}$ having support in the set $U$, then as $\lambda
\rightarrow\infty$ we get%
\begin{align*}
E_{\lambda}  &  \leq Q_{L}(u,u)+\lambda Q_{V}(u,u)\\
&  =Q_{L}(u,u)-\lambda\int_{U}|V|\left\vert u\right\vert ^{2}dm\rightarrow
-\infty
\end{align*}
as was claimed.
\end{proof}

\subsection{The positive spectrum}

We prove here criteria for positivity of the spectum of the operator $H=L+V$.
We use the notion of \emph{the} \emph{square of gradient }$\Gamma(u,v)$
defined as follows: for all $u,v\in\mathcal{D}$ we set
\begin{equation}
\Gamma(u,v):=\frac{1}{2}\left\{  uLv+vLu-L(uv)\right\}  . \label{sq of grad}%
\end{equation}
It is straightforward to show that the following identities hold true:
\begin{equation}
Q_{L}(u,v)=\int_{X}\Gamma(u,v)dm,\text{ } \label{sq of grad 1}%
\end{equation}

\begin{equation}
Q_{L}(uv,w)=\int_{X}v\Gamma(u,w)dm+\int_{X}u\Gamma(v,w)dm,
\label{sq of grad 2}%
\end{equation}%
\begin{align}
&  \int_{X}v\Gamma(u^{2},w)dm-2\int_{X}vu\Gamma(u,w)dm\label{sq of grad 3}\\
&  =\frac{1}{2}\int_{X\times X}\left(  u(y)-u(x)\right)  ^{2}\left(
w(y)-w(x)\right)  \left(  v(y)-v(x)\right)  J(x-y)dm(x)dm(y).\nonumber
\end{align}
Here $J(x-y)$ is the jump kernel associated with the (non-local) operator $L$,
see equations (\ref{Levy generator}) and (\ref{Levy symbol}). In particular,
we have%
\begin{align}
&  \int_{X}w\Gamma(u^{2},w)dm-2\int_{X}wu\Gamma(u,w)dm\label{sq of grad 4}\\
&  =\frac{1}{2}\int_{X\times X}\left(  u(y)-u(x)\right)  ^{2}\left(
w(y)-w(x)\right)  ^{2}J(x-y)dm(x)dm(y)\geq0.\nonumber
\end{align}
The identities listed above can be extended to the set of all bounded
functions $u,v$ and $w$ from $dom(Q_{L})$. We refer to \cite[Sec.
5]{Fukushima}.

The operator $(L,\mathcal{D})$ can be extended to each of the Banach spaces
$C_{\infty}(X)$ and $L^{q}(X,m),$ $1\leq q<\infty,$ as minus Markov generator.
The extended operators we denote $L_{\infty}$ and $L_{q}$ respectively.

\begin{theorem}
\label{negative spectrum}Assume that the quadratic form $Q=Q_{L}+Q_{V}$
defines a bounded below self-adjoint operator $H$ (see e.g. Theorem
\ref{L^p-L^infty theorem}). If there exists a function $0<f\in dom(L_{\infty
})$ such that the inequality
\[
V(x)\geq-\frac{Lf(x)}{f\left(  x\right)  }%
\]
holds almost everywhere, then $Spec(H)\subseteq\lbrack0,\infty).$

\begin{proof}
Let us assume first that $f$ is a locally constant function. Let us put
$W_{f}:=(-Lf)/f$ and let $\varphi\in\mathcal{D}$. If we put $\psi
:=\varphi/f\in\mathcal{D}$, then using equations (\ref{sq of grad 1}%
)-(\ref{sq of grad 4}) we get%
\begin{align*}
Q(\varphi,\varphi) &  =\int_{X}(\varphi L\varphi+V\varphi^{2})dm\geq\int
_{X}(\varphi L\varphi+W_{f}\varphi^{2})dm\\
&  =\int_{X}(L\varphi+W_{f}\varphi)\varphi dm=\int_{X}(\psi Lf-2\Gamma
(f,\psi)+fL\psi+W_{f}f\psi)f\psi dm.
\end{align*}
Since $Lf+W_{f}f=0$ the right-hand side $RHS$ of the inequality from above can
be written as%
\begin{align*}
RHS &  =\int_{X}(-2\psi f\Gamma(f,\psi)+f^{2}\psi L\psi)dm\\
&  =\int_{X}-2\psi f\Gamma(f,\psi)dm+Q_{L}(f^{2}\psi,\psi).
\end{align*}
It follows that%
\begin{align*}
Q(\varphi,\varphi) &  \geq\int_{X}-2\psi f\Gamma(f,\psi)dm+Q_{L}(f^{2}%
\psi,\psi)\\
&  =\int_{X}\{-2\psi f\Gamma(f,\psi)+f^{2}\Gamma(\psi,\psi)+\psi\Gamma
(f^{2},\psi)\}dm\\
&  =\int_{X}f^{2}\Gamma(\psi,\psi)dm+\int_{X}\{-2\psi f\Gamma(f,\psi
)+\psi\Gamma(f^{2},\psi)\}dm\\
&  \geq\int_{X}f^{2}\Gamma(\psi,\psi)dm\geq0.
\end{align*}
Thus $Q(\varphi,\varphi)\geq0$ for all $\varphi\in\mathcal{D}$. Since such
functions $\varphi$ form a core for $Q$, the result follows by an application
of the variational formula (\ref{variation formula}). 

In general one can choose a sequence of locally constant functions $f_{n}$
such that $W_{f_{n}}\rightarrow W_{f}$ locally uniformly in $X$, for instance
one can choose a $\delta$-sequence $\phi_{n}\in\mathcal{D}_{+}$ and set
$f_{n}:=f\ast\phi_{n}$. Then setting $\psi_{n}:=\varphi/f_{n}$ we get
\begin{align*}
Q(\varphi,\varphi)  & =\int_{X}(\varphi L\varphi+V\varphi^{2})dm\geq\int
_{X}(\varphi L\varphi+W_{f}\varphi^{2})dm\\
& =\lim_{n\rightarrow\infty}\int_{X}(\varphi L\varphi+W_{f_{n}}\varphi
^{2})dm\geq\limsup_{n\rightarrow\infty}\int_{X}f_{n}^{2}\Gamma(\psi_{n}%
,\psi_{n})dm\geq0.
\end{align*}
The proof of the theorem is finished.
\end{proof}
\end{theorem}

\begin{corollary}
\label{positive spectrum}Assume that $0<\alpha<1$ and that the following
inequality%
\[
V_{-}(x)\leq\left(  \Gamma_{p}\left(  \frac{1+\alpha}{2}\right)  \right)
^{2}\left\Vert x\right\Vert _{p}^{-\alpha}%
\]
holds almost everywhere, then
\[
Spec(\mathfrak{D}^{\alpha}+V)\subseteq\lbrack0,\infty).
\]

\end{corollary}

\begin{proof}
Let us set $\mathfrak{u}_{\beta}(x):=\left\Vert x\right\Vert _{p}^{\beta}$. By
\cite[Sec. 8.1, Eq. (1.6)]{Vladimirov94}, the function $\mathfrak{u}_{\beta}$
defines a distribution (a generalized function) which is holomorphic on
$\beta$ everywhere on the real line. The operator $\mathfrak{D}^{\alpha}%
:\psi\rightarrow\mathfrak{D}^{\alpha}\psi$ can be defined as convolution of
distributions $\mathfrak{u}_{-\alpha-1}/\Gamma_{p}(-\alpha)$ and $\psi,$ see
\cite[Sec. 9]{Vladimirov94}.

We claim that for all $\beta\neq\alpha$,
\begin{equation}
\frac{\mathfrak{D}^{\alpha}\mathfrak{u}_{\beta}}{\mathfrak{u}_{\beta}}%
=\frac{\Gamma_{p}(\beta+1)}{\Gamma_{p}(\beta+1-\alpha)}\mathfrak{u}_{-\alpha}.
\label{D^alpha/betta identity}%
\end{equation}
The case $\beta=0$ is trivial. For $\beta\neq0$ we apply the Fourier transform
argument. Remind that the Fourier transform $f\rightarrow\widehat{f}$ is a
linear isomorphism of $\mathcal{D}^{\prime}\rightarrow\mathcal{D}^{\prime}$.
By virtue of the results of \cite[Sec. 7.5]{Vladimirov94}, the equation
\begin{equation}
\widehat{\mathfrak{u}_{\gamma-1}}(\xi)=\Gamma_{p}(\gamma)\mathfrak{u}%
_{-\gamma}(\xi) \label{vladimirov identity}%
\end{equation}
holds true for all $\gamma\neq1$. Applying equation (\ref{vladimirov identity}%
) we obtain%
\begin{align*}
\widehat{\mathfrak{D}^{\alpha}\mathfrak{u}_{\beta}}(\xi)  &  =\mathfrak{u}%
_{\alpha}(\xi)\widehat{\mathfrak{u}_{\beta}}(\xi)=\mathfrak{u}_{\alpha}%
(\xi)\widehat{\mathfrak{u}_{\beta+1-1}}(\xi)\\
&  =\mathfrak{u}_{\alpha}(\xi)\Gamma_{p}(\beta+1)\mathfrak{u}_{-\beta-1}%
(\xi)=\Gamma_{p}(\beta+1)\mathfrak{u}_{-(1+\beta-\alpha)}(\xi)\\
&  =\frac{\Gamma_{p}(\beta+1)}{\Gamma_{p}(\beta+1-\alpha)}\Gamma_{p}%
(\beta+1-\alpha)\mathfrak{u}_{-(1+\beta-\alpha)}(\xi)\\
&  =\frac{\Gamma_{p}(\beta+1)}{\Gamma_{p}(\beta+1-\alpha)}\widehat
{\mathfrak{u}_{(1+\beta-\alpha)-1}}(\xi)=\frac{\Gamma_{p}(\beta+1)}{\Gamma
_{p}(\beta+1-\alpha)}\widehat{\mathfrak{u}_{\beta-\alpha}}(\xi),
\end{align*}
so by the unicity theorem the desired result follows.

For $\phi\in\mathcal{D}_{+}$ and $\beta:=(\alpha-1)/2$ we define the following
function%
\[
W_{\phi}:=\frac{\Gamma_{p}(\beta+1)}{\Gamma_{p}(\beta+1-\alpha)}%
\frac{\mathfrak{u}_{\beta-\alpha}\ast\phi}{\mathfrak{u}_{\beta}\ast\phi
}=\left(  \Gamma_{p}\left(  \frac{1+\alpha}{2}\right)  \right)  ^{2}%
\frac{\mathfrak{u}_{-\frac{1+\alpha}{2}}\ast\phi}{\mathfrak{u}_{-\frac
{1-\alpha}{2}}\ast\phi}.\text{ }%
\]
Equation (\ref{D^alpha/betta identity}) shows that $W_{\phi}$ belongs to
$C_{\infty}(X)$ and $W_{\phi}=Lf/f$ for some $0<f\in dom(L_{\infty})$, so
applying Theorem \ref{negative spectrum} we get
\[
Q_{W_{\phi}}(\varphi,\varphi)\leq Q_{\mathfrak{D}^{\alpha}}(\varphi,\varphi),
\]
for all $\varphi\in\mathcal{D}$. Let us choose a sequence $\{B_{n}%
:n=1,2,...\}$ of balls centred at the neutral element $0$ such that
$\cap_{n=1}^{\infty}B_{n}=\{0\}$ and set $\phi_{n}=1_{B_{n}}/m(B_{n})$.
Clearly $\phi_{n}\ast f$ converges to $f$ for any continuous function $f$,
whence
\[
W_{\phi_{n}}(x)\rightarrow W(x)=\left(  \Gamma_{p}\left(  \frac{1+\alpha}%
{2}\right)  \right)  ^{2}\frac{\mathfrak{u}_{-\frac{1+\alpha}{2}}%
(x)}{\mathfrak{u}_{-\frac{1-\alpha}{2}}(x)}=\left(  \Gamma_{p}\left(
\frac{1+\alpha}{2}\right)  \right)  ^{2}\left\Vert x\right\Vert _{p}^{-\alpha
}.
\]
Applying now Fatou lemma we conclude that for all $\varphi\in\mathcal{D}$,%
\[
Q_{W}(\varphi,\varphi)\leq Q_{\mathfrak{D}^{\alpha}}(\varphi,\varphi).
\]
It follows that for all $\varphi\in\mathcal{D}$,%
\[
-Q_{V}(\varphi,\varphi)\leq Q_{V_{-}}(\varphi,\varphi)\leq Q_{W}%
(\varphi,\varphi)\leq Q_{\mathfrak{D}^{\alpha}}(\varphi,\varphi),
\]
or equivalently,%
\[
Q(\varphi,\varphi):=Q_{\mathfrak{D}^{\alpha}}(\varphi,\varphi)+Q_{V}%
(\varphi,\varphi)\geq0.
\]
The set $\mathcal{D}$ forms a core for $Q(\varphi,\varphi)$, for reasongs
which depend upon which assumption we make on $V$, and the proof is completed
by an application of the variational formula (\ref{variation formula}).
\end{proof}

The following results show that the crucial issue for the existence of
negative eigenvalues in Theorem \ref{negative spec} for all $\lambda>0$ is the
rate at which the potential $V(x)$ converges to $0$ as $\left\Vert
x\right\Vert _{p}\rightarrow\infty$.

\begin{example}
\label{threshold}Let $0<\alpha<1$ and let $H_{\lambda}=\mathfrak{D}^{\alpha
}-\lambda V$ where $V(x)=(\left\Vert x\right\Vert _{p}+1)^{-\beta}$ for some
$0<\beta<1$ and $\lambda>0.$ If $\beta\geq\alpha$ then Theorem
\ref{negative spec} \ and Corollary \ref{positive spectrum} are applicable and
there exists a positive threshold for the existence of negative eigenvalues of
$H_{\lambda}$. If $0<\beta<\alpha$ the result is totally different.
\end{example}

\begin{theorem}
\label{no threshold} In the notation of Example \ref{threshold} assume that
$0<\beta<\alpha$, then $H_{\lambda}$ has non-empty negative spectrum for all
$\lambda>0$.
\end{theorem}

\begin{proof}
Let $f:=\mathfrak{D}^{-\alpha}1_{B}$ where $B$ is a ball centred at the
neutral element which we will specify later. The function $f$ \ belongs to
$\mathrm{dom}(\mathfrak{D}^{\alpha})$ and calculations based on the spectral
resolution formula and equation (\ref{Lambda-Phi eigenvalue}) show that%
\begin{align*}
\mathfrak{D}^{-\alpha}1_{B}/m(B)  &  =\mathfrak{D}^{-\alpha}\sum_{T:\text{
}B\subseteq T}f_{T}=\sum_{T:\text{ }B\subseteq T}\mathfrak{D}^{-\alpha}f_{T}\\
&  =\sum_{T:\text{ }B\subseteq T}\left(  \frac{m(T^{\prime})}{p}\right)
^{\alpha}f_{T}=\sum_{T:\text{ }B\subseteq T}m(T)^{\alpha}\left(  \frac{1_{T}%
}{m(T)}-\frac{1_{T^{\prime}}}{m(T^{\prime})}\right) \\
&  =m(B)^{\alpha-1}\sum_{T:\text{ }B\subseteq T}\left(  \frac{m(T)}%
{mB)}\right)  ^{\alpha-1}\left(  1_{T}-\frac{1}{p}1_{T^{\prime}}\right)
\text{. }%
\end{align*}
In particular, $W:=(\mathfrak{D}^{\alpha}f)/f$ is given by%
\[
W=\frac{1_{B}}{\mathfrak{D}^{-\alpha}1_{B}}=\frac{p-p^{\alpha}}{p-1}%
\frac{1_{B}}{m(B)^{\alpha}}=\frac{p-p^{\alpha}}{p-1}\frac{1_{B}}%
{\mathrm{diam}(B)^{\alpha}}.
\]
If $\lambda>0$ and $0<\beta<\alpha$, there exists a ball $B$ such that
$\mathrm{diam}(B)$ is large enough so that
\[
W(x)<\frac{\lambda}{(\left\Vert x\right\Vert _{p}+1)^{\beta}}=\lambda V(x)
\]
for all $x\in\mathbb{Q}_{p}$. Hence, as $f$ \ belongs to $\mathrm{dom}%
(\mathfrak{D}^{\alpha})$, we obtain%
\begin{align*}
Q_{\lambda}(f,f)  &  =Q_{\mathfrak{D}^{\alpha}}(f,f)-Q_{\lambda V}%
(f,f)<Q_{\mathfrak{D}^{\alpha}}(f,f)-Q_{W}(f,f)\\
&  =(\mathfrak{D}^{\alpha}f,f)-(W\cdot f,f)=0
\end{align*}
and an application of the Rayleigh-Ritz formulae yields the desired result.
\end{proof}

\section{An example}

In this section we study the quadratic form $Q_{H}(u,u)$ defined by the
Hamiltonian $H=L+V$. We provide our calculations assuming that $L=\mathfrak{D}%
^{\alpha}$ and $V(x)=b\left\Vert x\right\Vert _{p}^{-\alpha}$ for $0<\alpha<1$
and $b\geq b_{\ast}$, a critical value which will be specified later.

\subsection{The Dirichlet form}

We regard the function $h(x)=\left\Vert x\right\Vert _{p}^{\beta}$ as a
distribution, see \cite{Vladimirov94}. For $\beta\neq\alpha$ equation
(\ref{D^alpha/betta identity}) shows that (in the sence of distributions)%
\[
Lh(x)=\frac{\Gamma_{p}(\beta+1)}{\Gamma_{p}(\beta+1-\alpha)}\left\Vert
x\right\Vert _{p}^{\beta-\alpha}.
\]
In particular, for $\beta>\alpha-1$ the distributions $h(x)$ and $Lh(x)$ are
regular (generated by locally integrable functions) and the function%
\begin{equation}
V(x):=-\frac{Lh(x)}{h(x)}=-\frac{\Gamma_{p}(\beta+1)}{\Gamma_{p}%
(\beta+1-\alpha)}\left\Vert x\right\Vert _{p}^{-\alpha} \label{V_L_h eq}%
\end{equation}
belongs to $L_{loc}^{1}(X,m)$, so it defines a regular distribution as well.

\begin{theorem}
\label{Claim1-2}For $\alpha-1<\beta<\alpha$ the following statements hold true:

1. For $0<\beta<\alpha$ the function $V(x)$ is strictly positive and belongs
to $L_{loc}^{1}(X,m)$, so $H$ is a minus Markovian generator by Theorem
\ref{L^1_loc theorem}. Moreover, for any $b>0$ there exists $0<\beta<\alpha,$
a solution of the equation
\begin{equation}
-\frac{\Gamma_{p}(\beta+1)}{\Gamma_{p}(\beta+1-\alpha)}=b,
\label{Gamma_alpha_betta}%
\end{equation}
such that $V(x)=b\left\Vert x\right\Vert _{p}^{-\alpha}$ for this value of
$\beta$.

2. For $\alpha-1<\beta<0$ the function $V(x)$ is strictly negative, so $H$ is
not a minus Markovian generator. However, for these values of $\beta$%
\[
V_{-}(x)=-V(x)\leq\left(  \Gamma_{p}\left(  \frac{1+\alpha}{2}\right)
\right)  ^{2}\left\Vert x\right\Vert _{p}^{-\alpha},
\]
so $H$ is a non-negative definite operator by Corollary
\ref{positive spectrum}. Moreover, for any $0>b\geq b_{\ast}:=-\{\Gamma
_{p}\left(  (1+\alpha)/2\right)  \}^{2}$ there exist $\alpha-1<\beta_{1}%
\leq(\alpha-1)/2$ and $(\alpha-1)/2\leq\beta_{2}<0$, solutions of equation
(\ref{Gamma_alpha_betta}), such that $V(x)=b\left\Vert x\right\Vert
_{p}^{-\alpha}$ for these two values of $\beta$.
\end{theorem}

\begin{proof}
To prove the theorem we set $\vartheta=\beta+(1-\alpha)/2$ and write%
\[
-\frac{\Gamma_{p}(\beta+1)}{\Gamma_{p}(\beta+1-\alpha)}=-\Gamma_{p}\left(
\frac{1+\alpha}{2}+\vartheta\right)  \Gamma_{p}\left(  \frac{1+\alpha}%
{2}-\vartheta\right)  :=C_{\alpha}(\vartheta).
\]
The function $C_{\alpha}(\vartheta)$ is even, continuous and increasing on
each interval $[0,(1+\alpha)/2[$ and $](1+\alpha)/2,+\infty\lbrack$. Using the
very definition of the function $\Gamma_{p}(\xi)$ it is straightforward to
show that the following properties hold true:

\begin{enumerate}
\item $C_{\alpha}(0)=-\{\Gamma_{p}\left(  (1+\alpha)/2\right)  \}^{2}%
,C_{\alpha}((1-\alpha)/2)=0,$

\item $C_{\alpha}((1+\alpha)/2-0)=+\infty,C_{\alpha}((1+\alpha)/2+0)=-\infty,$

\item $C_{\alpha}(+\infty)=-p^{\alpha}<C_{\alpha}(0).$
\end{enumerate}

Clearly (1)-(3) imply the result. The proof of the theorem is finished.
\end{proof}

\bigskip

Let us choose $h(x)=\left\Vert x\right\Vert _{p}^{\beta}$ with $(\alpha
-1)/2<\beta<\alpha$. Then Theorem \ref{Claim1-2} applies, so $H=L+V$ is a
non-negative definite self-adjoint operator acting in $L^{2}(X,m)$.

According to our choice $h^{2}\in L_{loc}^{1}(X,m)$, so $h^{2}m$ is a Radon
measure. In particular, this allows us to define an isometry $U:L^{2}%
(X,h^{2}m)\rightarrow$ $L^{2}(X,m)$ by setting $U:g\rightarrow hg$. Consider a
non-negative self-adjoint operator%
\[
\mathcal{H}:=U^{-1}\circ H\circ U
\]
and let $Q_{\mathcal{H}}(u,u)=(\mathcal{H}^{1/2}u,\mathcal{H}^{1/2}u)$ be the
associated quadratic form. We have $Q_{H}=Q_{L}+Q_{V}$ whence
\begin{align*}
Q_{\mathcal{H}}(u,u)  &  =Q_{H}(hu,hu)=Q_{L}(hu,hu)+Q_{V}(hu,hu)\\
&  =\frac{1}{2}\int_{X}\int_{X}\left(  h(x)u(x)-h(y)u(y)\right)
^{2}J(x,y)dm(y)dm(x)\\
&  +\int_{X}V(x)u^{2}(x)h^{2}(x)dm(x)
\end{align*}
where
\begin{equation}
J(x,y)=-\frac{1}{\Gamma_{p}(-\alpha)}\frac{1}{\left\Vert x-y\right\Vert
_{p}^{1+\alpha}}\text{ }. \label{Jump-function}%
\end{equation}

\begin{theorem}
\label{H-DF}Assume that $(\alpha-1)/2<\beta<\alpha$. Then $Q_{\mathcal{H}%
}(u,u)$ is a Dirichlet form in $L^{2}(X,h^{2}m)$. Moreover, $\mathcal{D}%
\subset dom(Q_{\mathcal{H}})$ and for $u\in\mathcal{D}$,%
\begin{equation}
Q_{\mathcal{H}}(u,u)=\frac{1}{2}\int_{X}\int_{X}\left(  u(x)-u(y)\right)
^{2}J(x,y)h(y)dm(y)h(x)dm(x). \label{Q_H_identity}%
\end{equation}

\end{theorem}

\begin{proof}
Let us prove that $\mathcal{D}\subset dom(Q_{\mathcal{H}})$. It is enough to
show that $Q_{\mathcal{H}}(u,u)$ is finite for $u=\mathbf{1}_{B}$,
$B\in\mathcal{B}$. We have $Q_{\mathcal{H}}(u,u)=Q_{L}(hu,hu)+Q_{V}(hu,hu)$.
Since $V(x)=b\left\Vert x\right\Vert _{p}^{-\alpha}$ we get for $\beta
>(\alpha-1)/2$:
\[
|Q_{V}(hu,hu)|=|b|\int_{B}\left\Vert x\right\Vert _{p}^{-\alpha+2\beta
}dm(x)<\infty.
\]
Let us assume first that $0\notin B$, then clearly $hu\in\mathcal{D}.$ Since
$\mathcal{D}\subset dom(L)$,%
\[
Q_{L}(hu,hu)=(Lhu,hu)<\infty.
\]
Assume now that $0\in B$ and set $h_{B}:=h\mathbf{1}_{B}$, then%
\begin{align*}
Q_{L}(hu,hu)  &  =\frac{1}{2}%
%TCIMACRO{\diint }%
%BeginExpansion
{\displaystyle\iint}
%EndExpansion
\left(  h_{B}(x)-h_{B}(y)\right)  ^{2}J(x,y)dm(x)dm(y)\\
&  =%
%TCIMACRO{\diint \limits_{(x,y)\in B\times B:\text{ }\left\Vert x\right\Vert
%_{p}<\left\Vert y\right\Vert _{p}}}%
%BeginExpansion
{\displaystyle\iint\limits_{(x,y)\in B\times B:\text{ }\left\Vert x\right\Vert
_{p}<\left\Vert y\right\Vert _{p}}}
%EndExpansion
\left(  h(x)-h(y)\right)  ^{2}J(x,y)dm(x)dm(y)\\
&  +\int_{B}h^{2}(x)dm(x)\int_{B^{c}}J(x,y)dm(y).
\end{align*}
The second term, call it $II$, is finite. Indeed, we have
\[
II=\int_{B}h^{2}(x)dm(x)\int_{B^{c}}J(0,z)dm(z)<\infty.
\]
Without loss of generality we may assume that $diam(B)=1$. By the ultrametric
inequality, $\left\Vert x\right\Vert _{p}<\left\Vert y\right\Vert _{p}$
implies that $\left\Vert x-y\right\Vert _{p}=\left\Vert y\right\Vert _{p}$, so
the first term, call it $I$, can be estimated as follows:%
\begin{align*}
I  &  =-\frac{1}{\Gamma_{p}(-\alpha)}%
%TCIMACRO{\dsum \limits_{k=1}^{\infty}}%
%BeginExpansion
{\displaystyle\sum\limits_{k=1}^{\infty}}
%EndExpansion%
%TCIMACRO{\dsum \limits_{l=1}^{k}}%
%BeginExpansion
{\displaystyle\sum\limits_{l=1}^{k}}
%EndExpansion%
%TCIMACRO{\dint \limits_{\left\Vert x\right\Vert _{p}=p^{-k}}}%
%BeginExpansion
{\displaystyle\int\limits_{\left\Vert x\right\Vert _{p}=p^{-k}}}
%EndExpansion
dm(x)%
%TCIMACRO{\dint \limits_{\left\Vert y\right\Vert _{p}=p^{-k+l}}}%
%BeginExpansion
{\displaystyle\int\limits_{\left\Vert y\right\Vert _{p}=p^{-k+l}}}
%EndExpansion
dm(y)\left(  \left\Vert x\right\Vert _{p}^{\beta}-\left\Vert y\right\Vert
_{p}^{\beta}\right)  ^{2}\left\Vert y\right\Vert _{p}^{-(1+\alpha)}\\
&  =-\frac{1}{\Gamma_{p}(-\alpha)}\left(  1-\frac{1}{p}\right)  ^{2}%
%TCIMACRO{\dsum \limits_{k=1}^{\infty}}%
%BeginExpansion
{\displaystyle\sum\limits_{k=1}^{\infty}}
%EndExpansion%
%TCIMACRO{\dsum \limits_{l=1}^{k}}%
%BeginExpansion
{\displaystyle\sum\limits_{l=1}^{k}}
%EndExpansion
p^{-k}p^{-k+l}p^{-(1+\alpha)(-k+l)}\left(  p^{-k\beta}-p^{(-k+l)\beta}\right)
^{2}\\
&  =-\frac{1}{\Gamma_{p}(-\alpha)}\left(  1-\frac{1}{p}\right)  ^{2}%
%TCIMACRO{\dsum \limits_{k=1}^{\infty}}%
%BeginExpansion
{\displaystyle\sum\limits_{k=1}^{\infty}}
%EndExpansion
p^{-k(1-\alpha+2\beta)}%
%TCIMACRO{\dsum \limits_{l=1}^{k}}%
%BeginExpansion
{\displaystyle\sum\limits_{l=1}^{k}}
%EndExpansion
p^{-l\alpha}\left(  1-p^{l\beta}\right)  ^{2}.
\end{align*}
That $I$ is finite for $(\alpha-1)/2<\beta<\alpha$ follows by inspection.

Since the function $u\rightarrow Q_{\mathcal{H}}(u,u)$ is lower
semi-continuous, equation (\ref{Q_H_identity}) is enough to prove for
$u=\mathbf{1}_{B}$ where $B$ is a ball such that $0\notin B$. In this case the
function $h_{B}=h\mathbf{1}_{B}$ belongs to $\mathcal{D}$. Let us consider the
distribution $f_{\gamma}(x)=\left\Vert x\right\Vert _{p}^{\gamma-1}/\Gamma
_{p}(\gamma)$. According to \cite[Section IX]{Vladimirov94}, $h(x)=\Gamma
_{p}(\beta+1)f_{\beta+1}$ and $-Lh=f_{-\alpha}\ast\Gamma_{p}(\beta
+1)f_{\beta+1}$ whence, setting $C:=\Gamma_{p}(\beta+1)$, we get
\begin{align*}
Q_{V}(hu,hu)  &  =\int(-Lh)h_{B}dm=\left(  (-Lh)\ast h_{B}\right)  (0)\\
&  =C(\left(  f_{-\alpha}\ast f_{\beta+1}\right)  \ast h_{B})(0)=C((f_{\beta
+1}\ast(f_{-\alpha}\ast h_{B}))(0)\\
&  =\int h(-Lh_{B})dm=-%
%TCIMACRO{\diint }%
%BeginExpansion
{\displaystyle\iint}
%EndExpansion
(h_{B}(x)-h_{B}(y))h(x)J(x,y)dm(x)dm(y)
\end{align*}
and by symmetry%
\[
Q_{V}(hu,hu)=-%
%TCIMACRO{\diint }%
%BeginExpansion
{\displaystyle\iint}
%EndExpansion
(h_{B}(y)-h_{B}(x))h(y)J(x,y)dm(x)dm(y).
\]
Thus finally we get%
\begin{equation}
Q_{V}(hu,hu)=-\frac{1}{2}%
%TCIMACRO{\diint }%
%BeginExpansion
{\displaystyle\iint}
%EndExpansion
(h_{B}(x)-h_{B}(y))(h(x)-h(y))J(x,y)dm(x)dm(y). \label{Q_V}%
\end{equation}
On the other hand, for $u$ as above,%
\begin{equation}
Q_{L}(hu,hu)=\frac{1}{2}%
%TCIMACRO{\diint }%
%BeginExpansion
{\displaystyle\iint}
%EndExpansion
(h_{B}(x)-h_{B}(y))^{2}J(x,y)dm(x)dm(y). \label{Q_L}%
\end{equation}
Clearly equations (\ref{Q_V}) and (\ref{Q_L}) yield equation
(\ref{Q_H_identity}).

Thus, the quadratic form $Q_{\mathcal{H}}(u,u)$ is densly defined, closed,
non-negative definite, and Markovian. That means that $Q_{\mathcal{H}}(u,u)$
is a Dirichlet form in $L^{2}(X,h^{2}m)$ as claimed.
\end{proof}

\begin{definition}
A Dirichlet form $Q(u,u)$ relative to $L^{2}(X,\mu)$ (respectively, a
symmetric Markovian semigroup $(P_{t})_{t>0}$ in $L^{2}(X,\mu)$) is called
transient if the associated resolvent $(G_{\lambda})_{\lambda>0}$ can be
extended for the value $\lambda=0$ as a self-adjoint (possibly unbounded)
operator $G_{0}=\int_{0}^{\infty}P_{t}dt$ such that $\mathbf{1}_{K}\in
dom(G_{0})$ for every compact set $K\subset X$.
\end{definition}

One can show that the form $Q(u,u)$ is transient if and only if the following
condition holds: for every compact set $K\subset X$ there exists a constant
$C_{K}>0$ such that%
\[
\int_{X}\left\vert u\right\vert d\mu\leq C_{K}\sqrt{Q(u,u)},\text{ }\forall
u\in dom(Q)\text{ }\footnote{This condition of transience was first introduced
by A. Beurling and J. Deny in the unreplacable paper \cite{BeurlingDeny}. It
is slightly more restrictive than the definition of transience given in
\cite[Section 1.5]{Fukushima}.}\text{.}%
\]

\begin{theorem}
\label{non-hom.DF}In the setting of Theorem \ref{H-DF}:

\begin{enumerate}
\item There exists a hierarhical Laplacian $\mathcal{L}$, related to the
(non-homogeneous) ultrametric measure space $(X,hm)$, such that
\[
Q_{\mathcal{H}}(u,u)=Q_{\mathcal{L}}(u,u),\text{ }\forall u\in L^{2}(X,hm)\cap
L^{2}(X,h^{2}m).
\]

\item $\mathcal{D}\subset dom(Q_{\mathcal{L}})$ is a core of $Q_{\mathcal{L}}$
(i.e. $Q_{\mathcal{L}}(u,u)$ is a regular Dirichlet form in $L^{2}(X,hm)$).

\item The Dirichlet form $Q_{\mathcal{L}}$ relative to $L^{2}(X,hm)$ is
transient. In particular, the Dirichlet form $Q_{\mathcal{H}}$ relative to
$L^{2}(X,h^{2}m)$ is transient as well.
\end{enumerate}
\end{theorem}

\begin{proof}
Consider the function%
\[
J(B):=-\frac{1}{\Gamma_{p}(-\alpha)}\frac{1}{m(B)^{1+\alpha}},\text{ }%
B\in\mathcal{B}\text{,}%
\]
defined on the set $\mathcal{B}$ of all open balls. Since in the $p$-adic
metric $m(B)=diam(B)$ for any ball $B$, we get
\[
J(x,y)=J(x\curlywedge y)
\]
where $x\curlywedge y$ is the minimal ball which contains $x$ and $y$.
Consider also the Radon measure $\widetilde{m}=hm$. We claim that the
following properties hold true:

\begin{description}
\item[(i)] $S\subset T\Longrightarrow J(S)>J(T)$ and $J(T)\rightarrow0$ as
$T\rightarrow X$.

\item[(ii)] $\widetilde{\lambda}(B):=\sum_{S:\text{ }B\subseteq S}%
\widetilde{m}(S)\left(  J(S)-J(S^{\prime})\right)  <\infty$ for any
$B\in\mathcal{B}.$

\item[(iii)] $\widetilde{\lambda}(B)\rightarrow+\infty$ as $B\rightarrow\{x\}$
for any $x\in X.$
\end{description}

The property $(i)$ is evident. To prove $(ii)$ we write%
\begin{align*}
\widetilde{\lambda}(B)  &  =-\frac{1}{\Gamma_{p}(-\alpha)}\left(  1-\frac
{1}{p^{1+\alpha}}\right)  \sum_{S:\text{ }B\subseteq S}\frac{\widetilde{m}%
(S)}{m(S)^{1+\alpha}}\\
&  =(p^{\alpha}-1)\sum_{S:\text{ }B\subseteq S}\frac{\widetilde{m}%
(S)}{m(S)^{1+\alpha}}.
\end{align*}
Next, using the identity%
\[
\int f(\left\Vert x\right\Vert _{p})dm(x)=\left(  1-\frac{1}{p}\right)
\sum_{\gamma=-\infty}^{\infty}f(p^{\gamma})p^{\gamma},
\]
we obtain that if $0\in$ $S$ then
\begin{equation}
\widetilde{m}(S)=\frac{p-1}{p-p^{-\beta}}m(S)^{1+\beta}\text{, }
\label{ass.eq}%
\end{equation}
so
\begin{equation}
\frac{\widetilde{m}(S)}{m(S)^{1+\alpha}}=\frac{p-1}{p-p^{-\beta}}\frac
{1}{m(S)^{\alpha-\beta}}.\text{ } \label{ass.eq'}%
\end{equation}
Clearly equality (\ref{ass.eq'}) implies $(ii)$. On the other hand, for
$B\in\mathcal{B}(x)$ small enough we have
\begin{equation}
\widetilde{\lambda}(B)\geq(p^{\alpha}-1)\frac{\widetilde{m}(B)}{m(B)^{1+\alpha
}}>(p^{\alpha}-1)m(B)^{-\alpha}\min_{y\in B}\left\Vert y\right\Vert
_{p}^{\beta} \label{ass.ineq.}%
\end{equation}
and%
\begin{equation}
\min_{y\in B}\left\Vert y\right\Vert _{p}^{\beta}=\left\{
\begin{array}
[c]{ccc}%
\left\Vert x\right\Vert _{p}^{\beta} & \text{if} & x\neq0\\
(m(B)^{\beta} & \text{if} & x=0
\end{array}
\right.  , \label{ass.B}%
\end{equation}
so (\ref{ass.ineq.}) and (\ref{ass.B}) imply $(iii)$.

According to \cite[Section 2]{Ben}, properties $(i)-(iii)$ imply that the
operator
\begin{equation}
\mathcal{L}u(x)=\int\left(  u(x)-u(y)\right)  J(x,y)d\widetilde{m}(y)
\label{L_cal}%
\end{equation}
is a hierarhical Laplacian in $L^{2}(X,\widetilde{m})$. In particular,
$\mathcal{D}\subset dom(\mathcal{L})$ and for $u\in\mathcal{D}$ we have
\[
Q_{\mathcal{L}}(u,u)=\frac{1}{2}\int\int\left(  u(x)-u(y)\right)
^{2}J(x,y)d\widetilde{m}(y)d\widetilde{m}(x)=Q_{\mathcal{H}}(u,u).
\]
That $\mathcal{D}$ is a core of $Q_{\mathcal{L}}$ follows from the fact that
$\mathcal{L}$, as a hierarchical Laplacian, is essentially self-adjoint.
Indeed, in this case $(Q_{\mathcal{L}},dom(Q_{\mathcal{L}}))$ coinsides with
the minimal extension of $(Q_{\mathcal{L}},\mathcal{D})$ which has
$\mathcal{D}$ as a core.

The proof of the fact that the Markovian semigroup $(e^{-t\mathcal{L}})_{t>0}$
is transient, i.e. that $\mathbf{1}_{K}$ belongs to $dom(G_{0})$ for any
compact set $K$, we postpone to the next section (Theorem \ref{green_function}%
). Let us show how to derive the Beurling-Deny condition of transience from
the transience of the semigroup $(e^{-t\mathcal{L}})_{t>0}$. For any $u\in
dom(Q_{\mathcal{L}})$ we have $|u|\in dom(Q_{\mathcal{L}})$ and
$Q_{\mathcal{L}}(|u|,|u|)\leq Q_{\mathcal{L}}(u,u).$ Also $v:=G_{0}%
\mathbf{1}_{K}$ is in $dom(\mathcal{L})$ and $\mathcal{L}v=\mathbf{1}_{K}$
whence
\begin{align*}
\int_{K}\left\vert u\right\vert d\widetilde{m}  &  =Q_{\mathcal{L}}(|u|,v)\\
&  \leq\sqrt{Q_{\mathcal{L}}(v,v)}\sqrt{Q_{\mathcal{L}}(u,u)}.
\end{align*}
Setting $C_{K}:=\sqrt{Q_{\mathcal{L}}(v,v)}$ we get the desired result. The
proof is finished.
\end{proof}

\subsection{The Green function $g_{\mathcal{L}}(x,y)$}

In what follows we assume that $(\alpha-1)/2<\beta<\alpha$. The Markovian
resolvent $G_{\lambda}=(\mathcal{L}+\lambda\mathrm{I})^{-1}$, $\lambda>0$,
acts in Banach spaces $C_{\infty}(X)$ and $L^{p}(X,\widetilde{m})$, where
$\widetilde{m}=hm$, as a bounded operator and admits the following
representation
\[
G_{\lambda}u(x)=\int g_{\mathcal{L}}(\lambda,x,y)u(y)d\widetilde{m}(y).
\]
Here $g_{\mathcal{L}}(\lambda,x,y)$, the so called $\lambda$-Green function,
is a continuous function taking finite values outside the diagonal set. As a
function of $\lambda$ it decreases, so the limit (finite or infinite)%
\[
g_{\mathcal{L}}(x,y):=\lim_{\lambda\rightarrow0}g_{\mathcal{L}}(\lambda,x,y)
\]
exists. The function $g_{\mathcal{L}}(x,y)$ is called the Green function of
the operator $\mathcal{L}$.

\begin{theorem}
\label{green_function}The Green function $g_{\mathcal{L}}(x,y)$ is a
continuous function taking finite values off the diagonal set. Moreover, the
following relationship holds:
\begin{equation}
g_{\mathcal{L}}(x,y)\asymp\frac{\left\Vert x-y\right\Vert _{p}^{\alpha-1}%
}{\left(  \left\Vert x\right\Vert _{p}\vee\left\Vert y\right\Vert _{p}\right)
^{2\beta}}, \label{g_cal_L_as}%
\end{equation}
or equivalently%
\begin{equation}
\frac{g_{\mathcal{L}}(x,y)}{g_{L}(x,y)}\asymp\left(  \frac{1}{\left\Vert
x\right\Vert _{p}}\wedge\frac{1}{\left\Vert y\right\Vert _{p}}\right)
^{2\beta}. \label{g_cal_L_g_L}%
\end{equation}

\end{theorem}

\begin{proof}
Let us assume that $X$ is equipped with the ultrametric $d(x,y)=p^{-\alpha
}\left\Vert x-y\right\Vert _{p}^{\alpha}$, intrinsic for the hierarchical
Laplacian $L$, and define the following variables%
\[
F(x,R)=\left(  \int_{R}^{\infty}\left(  \frac{1}{m(B_{r}(x))}\int_{B_{r}%
(x)}hdm\right)  \frac{dr}{r^{2}}\right)  ^{-1}%
\]
and%
\begin{equation}
\widetilde{d}(x,y)=F(x,d(x,y)). \label{d^tilda'}%
\end{equation}
Since for each fixed $x$ the function $R\rightarrow F(x,R)$ is continuous,
strictly increasing, $0$ at $0$ and $\infty$ at $\infty$, $\widetilde{d}(x,y)$
is an ultrametric on $X$. Let $\widetilde{B}_{\widetilde{R}}(x)$ be a
$\widetilde{d}$-ball of radius $\widetilde{R}$ centred at $x$. Then
$\widetilde{B}_{\widetilde{R}}(x)=B_{R}(x)$ whenever%
\[
\widetilde{R}=F(x,R).
\]
Since $L$ is a hierarchical Laplacian acting in $L^{2}(X,m)$ and $d(x,y)$ is
its intrinsic ultrametric, we have (see \cite[equation (3.11)]{BGPW})%
\begin{align}
J(x,y)  &  =\int_{d(x,y)}^{\infty}\frac{1}{m(B_{R}(x))}\frac{dR}{R^{2}%
}\label{J-function'}\\
&  =\int_{\widetilde{d}(x,y)}^{\infty}\frac{1}{\widetilde{m}(\widetilde
{B}_{\widetilde{R}}(x))}\frac{d\widetilde{R}}{\widetilde{R}^{2}}.\nonumber
\end{align}
It follows that $\widetilde{d}(x,y)$ is intrinsic ultrametric corresponding to
the hierarchical Laplacian $\mathcal{L}$ and%
\begin{align*}
\widetilde{V}(x,\widetilde{R})  &  :=\widetilde{m}(\widetilde{B}%
_{\widetilde{R}}(x))\\
&  =\widetilde{m}(B_{R}(x))=\int_{B_{R}(x)}hdm
\end{align*}
is its volume-function. We claim that%
\begin{equation}
\frac{\widetilde{m}(B_{R}(x))}{m(B_{R}(x))}\asymp\left\{
\begin{array}
[c]{ccc}%
m(B_{R}(x))^{\beta} & \text{if} & d(0,x)\leq R\\
h(x) & \text{if} & d(0,x)>R
\end{array}
\right.  . \label{Claim_d_R}%
\end{equation}
Indeed, if $d(0,x)\leq R$ then $B_{R}(x)=B_{R}(0)$, so applying (\ref{ass.eq}%
), we get
\begin{align*}
\frac{\widetilde{m}(B_{R}(x))}{m(B_{R}(x))}  &  =\frac{1}{m(B_{R}(x))}%
\int_{B_{R}(x)}hdm\\
&  =\frac{1}{m(B_{R}(0))}\int_{B_{R}(0)}hdm\\
&  =\frac{p-1}{p-p^{-\beta}}m(B_{R}(0))^{\beta}=\frac{p-1}{p-p^{-\beta}%
}m(B_{R}(x))^{\beta}.
\end{align*}
On the other hand, if $d(0,x)>R$ then from $y\in B_{R}(x)$ we get that
$d(y,0)=d(x,0),$ so%
\begin{align*}
\frac{\widetilde{m}(B_{R}(x))}{m(B_{R}(x))}  &  =\frac{1}{m(B_{R}(x))}%
\int_{B_{R}(x)}h(y)dm(y)\\
&  =\frac{1}{m(B_{R}(x))}\int_{B_{R}(x)}h(x)dm(y)=h(x).
\end{align*}
Notice that asymptotic relationship (\ref{Claim_d_R}) holds uniformly in $x$
and $R$ in the sence that the corresponding two sided inequality contains
constants which do not depend on $x$ and $R$. In turn, (\ref{Claim_d_R})
implies the following (uniform) asymptotic relationship:%
\begin{equation}
\widetilde{R}=F(x,R)\asymp\left\{
\begin{array}
[c]{ccc}%
R/h(x) & \text{if} & R<d(0,x)\\
R^{\frac{\alpha-\beta}{\alpha}} & \text{if} & R\geq d(0,x)
\end{array}
\right.  \label{Claim_tilda_R}%
\end{equation}
Indeed, if $d(0,x)\leq R$ then
\begin{align*}
\int_{R}^{\infty}\frac{\widetilde{m}(B_{r}(x))}{m(B_{r}(x))}\frac{dr}{r^{2}}
&  \asymp\int_{R}^{\infty}m(B_{r}(x))^{\beta}\frac{dr}{r^{2}}\\
&  \asymp\int_{R}^{\infty}r^{-\left(  2-\frac{\beta}{\alpha}\right)  }dr\asymp
R^{-\left(  1-\frac{\beta}{\alpha}\right)  },
\end{align*}
so
\[
\widetilde{R}:=F(x,R)\asymp R^{1-\frac{\beta}{\alpha}}.
\]
If $d(0,x)\geq R$ then for some constants $C_{1},C_{2}>0$,%
\begin{align*}
\int_{R}^{\infty}\frac{\widetilde{m}(B_{r}(x))}{m(B_{r}(x))}\frac{dr}{r^{2}}
&  =\int_{R}^{d(0,x)}\frac{\widetilde{m}(B_{r}(x))}{m(B_{r}(x))}\frac
{dr}{r^{2}}+\int_{d(0,x)}^{\infty}\frac{\widetilde{m}(B_{r}(x))}{m(B_{r}%
(x))}\frac{dr}{r^{2}}\\
&  =C_{1}d(0,x)^{\frac{\beta}{\alpha}}\left(  \frac{1}{R}-\frac{1}%
{d(0,x)}\right)  +\frac{C_{2}}{d(0,x)^{1-\frac{\beta}{\alpha}}}\\
&  \asymp\frac{d(0,x)^{\frac{\beta}{\alpha}}}{R}\asymp\frac{h(x)}{R},
\end{align*}
so%
\[
\widetilde{R}:=F(x,R)\asymp\frac{R}{h(x)}.
\]
Furthermore, asymptotic relationships (\ref{Claim_d_R}) and
(\ref{Claim_tilda_R}) yield the following (uniform) asymptotic relationship%
\begin{align}
\widetilde{V}(x,\widetilde{R})  &  =\widetilde{m}(B_{R}(x))\label{cal_V}\\
&  \asymp\left\{
\begin{array}
[c]{ccc}%
h(x)R^{\frac{1}{\alpha}} & \text{if} & R<d(0,x)\\
R^{\frac{1+\beta}{\alpha}} & \text{if} & R\geq d(0,x)
\end{array}
\right.  ,\nonumber
\end{align}
or equivalently, we get%
\begin{equation}
\widetilde{V}(x,\widetilde{R})\asymp\left\{
\begin{array}
[c]{ccc}%
h(x)^{1+\frac{1}{\alpha}}\widetilde{R}^{\frac{1}{\alpha}} & \text{if} &
\widetilde{R}<\widetilde{d}(0,x)\\
\widetilde{R}^{\frac{1+\beta}{\alpha-\beta}} & \text{if} & \widetilde{R}%
\geq\widetilde{d}(0,x)
\end{array}
\right.  . \label{cal_V_cal_R}%
\end{equation}
1. Let us consider the case $\left\Vert x-y\right\Vert _{p}=\left\Vert
x\right\Vert _{p}\vee\left\Vert y\right\Vert _{p}$. Then clearly
$d(x,y)=d(0,x)\vee d(0,y)$, and similar equation holds in $\widetilde{d}$
metric. If $R\geq d(0,x)$ then
\begin{equation}
\widetilde{R}:=F(x,R)\asymp R^{1-\frac{\beta}{\alpha}} \label{cal_d'}%
\end{equation}
and
\begin{equation}
\widetilde{V}(x,\widetilde{R})\asymp R^{\frac{1+\beta}{\alpha}}\asymp
\widetilde{R}^{\frac{1+\beta}{\alpha-\beta}}, \label{cal_L_volume}%
\end{equation}
Equation (\ref{cal_L_volume}) implies the following two results:

\begin{enumerate}
\item Since $\delta:=\frac{1+\beta}{\alpha-\beta}>1$, the function
$\widetilde{R}\rightarrow1/\widetilde{V}(x,\widetilde{R})$ is integrable at
$\infty$ for any fixed $x$, so the Markovian semigroup $(e^{-t\mathcal{L}%
})_{t>0}$ (equivalently, the Dirichlet form $Q_{\mathcal{L}}$) is transient
(see \cite[Theorem 2.28]{BGPW}) as it has been stated in Theorem
\ref{non-hom.DF}.

\item The fact that $\widetilde{V}(x,\widetilde{R})\asymp\widetilde{R}%
^{\delta}$, $\delta>1$, for $\widetilde{R}\geq$ $\widetilde{d}(0,x)$, yield
the following asymptotic relationship%
\begin{equation}
g_{\mathcal{L}}(x,y)=%
%TCIMACRO{\dint \limits_{\widetilde{d}(x,y)}^{\infty}}%
%BeginExpansion
{\displaystyle\int\limits_{\widetilde{d}(x,y)}^{\infty}}
%EndExpansion
\frac{d\widetilde{R}}{\widetilde{V}(x,\widetilde{R})}\asymp\frac{\widetilde
{d}(x,y)}{\widetilde{V}(x,\widetilde{d}(x,y))}, \label{g_cal_L''}%
\end{equation}
or equivalently, see equations (\ref{cal_d'}) and (\ref{cal_L_volume}),
\end{enumerate}

\begin{equation}
g_{\mathcal{L}}(x,y)\asymp\left\Vert x-y\right\Vert _{p}^{\alpha-1-2\beta
}=\frac{\left\Vert x-y\right\Vert _{p}^{\alpha-1}}{\left(  \left\Vert
x\right\Vert _{p}\vee\left\Vert y\right\Vert _{p}\right)  ^{2\beta}}\text{ }
\label{g_cal_L'}%
\end{equation}
provided $\left\Vert x\right\Vert _{p}\leq\left\Vert x-y\right\Vert _{p}$.
Similarly, by symmetry, relationship (\ref{g_cal_L'}) holds provided
$\left\Vert y\right\Vert _{p}\leq\left\Vert x-y\right\Vert _{p}$. Thus
finally, the assumption $\left\Vert x-y\right\Vert _{p}=\left\Vert
x\right\Vert _{p}\vee\left\Vert y\right\Vert _{p}$ implies (\ref{g_cal_L'}),
as it was claimed.

2. Let us consider the case $\left\Vert x-y\right\Vert _{p}<\left\Vert
x\right\Vert _{p}\vee\left\Vert y\right\Vert _{p}$. In this case we have:
$\left\Vert x\right\Vert _{p}=\left\Vert y\right\Vert _{p}$ and $\left\Vert
x-y\right\Vert _{p}<\left\Vert x\right\Vert _{p}$, similar relations hold in
$d$ and $\widetilde{d}$ metrics. Having this in mind we write
\[
g_{\mathcal{L}}(x,y)=%
%TCIMACRO{\dint \limits_{\widetilde{d}(x,y)}^{\infty}}%
%BeginExpansion
{\displaystyle\int\limits_{\widetilde{d}(x,y)}^{\infty}}
%EndExpansion
\frac{d\widetilde{R}}{\widetilde{V}(x,\widetilde{R})}=\left(
%TCIMACRO{\dint \limits_{\widetilde{d}(x,y)}^{\widetilde{d}(0,x)}}%
%BeginExpansion
{\displaystyle\int\limits_{\widetilde{d}(x,y)}^{\widetilde{d}(0,x)}}
%EndExpansion
+%
%TCIMACRO{\dint \limits_{\widetilde{d}(0,x)}^{\infty}}%
%BeginExpansion
{\displaystyle\int\limits_{\widetilde{d}(0,x)}^{\infty}}
%EndExpansion
\right)  \frac{d\widetilde{R}}{\widetilde{V}(x,\widetilde{R})}=I+II.
\]
Since $\widetilde{d}(0,x)\leq\widetilde{R}$ implies $\widetilde{V}%
(x,\widetilde{R})\asymp\widetilde{R}^{\frac{1+\beta}{\alpha-\beta}}$, we get
\[
II\asymp\frac{\widetilde{d}(0,x)}{\widetilde{V}(x,\widetilde{d}(0,x))}%
\asymp\frac{1}{\widetilde{d}(0,x)^{\frac{1-\alpha+2\beta}{\alpha-\beta}}}.
\]
To estimate the first term we write
\[
I=%
%TCIMACRO{\dint \limits_{\widetilde{d}(x,y)}^{\widetilde{d}(0,x)}}%
%BeginExpansion
{\displaystyle\int\limits_{\widetilde{d}(x,y)}^{\widetilde{d}(0,x)}}
%EndExpansion
\frac{d\widetilde{R}}{\widetilde{V}(x,\widetilde{R})}\asymp\frac
{1}{h(x)^{1+\frac{1}{\alpha}}}%
%TCIMACRO{\dint \limits_{\widetilde{d}(x,y)}^{\widetilde{d}(0,x)}}%
%BeginExpansion
{\displaystyle\int\limits_{\widetilde{d}(x,y)}^{\widetilde{d}(0,x)}}
%EndExpansion
\frac{d\widetilde{R}}{\widetilde{R}^{\frac{1}{\alpha}}}%
\]
and%
\begin{align*}
\frac{1}{h(x)^{1+\frac{1}{\alpha}}}%
%TCIMACRO{\dint \limits_{\widetilde{d}(x,y)}^{\widetilde{d}(0,x)}}%
%BeginExpansion
{\displaystyle\int\limits_{\widetilde{d}(x,y)}^{\widetilde{d}(0,x)}}
%EndExpansion
\frac{d\widetilde{R}}{\widetilde{R}^{\frac{1}{\alpha}}}  &  =\frac
{1}{h(x)^{1+\frac{1}{\alpha}}}\left(  \frac{1}{\widetilde{d}(x,y)^{\frac
{1}{\alpha}-1}}-\frac{1}{\widetilde{d}(0,x)^{\frac{1}{\alpha}-1}}\right) \\
&  =\frac{\widetilde{d}(x,y)^{1-\frac{1}{\alpha}}}{h(x)^{1+\frac{1}{\alpha}}%
}\left(  1-\left(  \frac{\widetilde{d}(x,y)}{\widetilde{d}(0,x)}\right)
^{\frac{1}{\alpha}-1}\right)  .
\end{align*}
Finally, since $\left\Vert x\right\Vert _{p}=\left\Vert y\right\Vert _{p}$ and
$\left\Vert x-y\right\Vert _{p}<\left\Vert x\right\Vert _{p}$, we have%
\begin{align*}
g_{\mathcal{L}}(x,y)  &  =I+II\\
&  \asymp\frac{\widetilde{d}(x,y)^{1-\frac{1}{\alpha}}}{h(x)^{1+\frac
{1}{\alpha}}}\left(  1-\left(  \frac{\widetilde{d}(x,y)}{\widetilde{d}%
(0,x)}\right)  ^{\frac{1}{\alpha}-1}\right)  +\frac{1}{\widetilde
{d}(0,x)^{\frac{1-\alpha+2\beta}{\alpha-\beta}}}\\
&  =\frac{\widetilde{d}(x,y)^{1-\frac{1}{\alpha}}}{h(x)^{1+\frac{1}{\alpha}}%
}\left(  \left(  1-\left(  \frac{\widetilde{d}(x,y)}{\widetilde{d}%
(0,x)}\right)  ^{\frac{1}{\alpha}-1}\right)  +\frac{\widetilde{d}%
(x,y)^{\frac{1}{\alpha}-1}h(x)^{1+\frac{1}{\alpha}}}{\widetilde{d}%
(0,x)^{\frac{1-\alpha+2\beta}{\alpha-\beta}}}\right)  .
\end{align*}
According to (\ref{cal_d'}) $h(x)\asymp\widetilde{d}(0,x)^{\frac{\beta}%
{\alpha-\beta}}$ whence%
\[
\frac{h(x)^{1+\frac{1}{\alpha}}}{\widetilde{d}(0,x)^{\frac{1-\alpha+2\beta
}{\alpha-\beta}}}\asymp\frac{\widetilde{d}(0,x)^{\frac{\beta}{\alpha-\beta
}\left(  1+\frac{1}{\alpha}\right)  }}{\widetilde{d}(0,x)^{\frac
{1-\alpha+2\beta}{\alpha-\beta}}}\asymp\frac{1}{\widetilde{d}(0,x)^{\frac
{1}{\alpha}-1}}%
\]
and thus, using (\ref{Claim_tilda_R}), we get%
\begin{align*}
g_{\mathcal{L}}(x,y)  &  \asymp\frac{\widetilde{d}(x,y)^{1-\frac{1}{\alpha}}%
}{h(x)^{1+\frac{1}{\alpha}}}\asymp\left(  \frac{d(x,y)}{h(x)}\right)
^{1-\frac{1}{\alpha}}\frac{1}{h(x)^{1+\frac{1}{\alpha}}}\\
&  =\frac{d(x,y)^{1-\frac{1}{\alpha}}}{h(x)^{2}}\asymp\frac{\left\Vert
x-y\right\Vert _{p}^{\alpha-1}}{\left\Vert x\right\Vert _{p}^{2\beta}}%
=\frac{\left\Vert x-y\right\Vert _{p}^{\alpha-1}}{\left(  \left\Vert
x\right\Vert _{p}\vee\left\Vert y\right\Vert _{p}\right)  ^{2\beta}}.
\end{align*}
The proof of the theorem is finished.
\end{proof}

\subsection{The Green function $g_{H}(x,y)$}

Throughout this section we assume that $(\alpha-1)/2\leq\beta<\alpha$ and that
$b$ is a solution of equation (\ref{Gamma_alpha_betta}). Then, by Theorem
\ref{Claim1-2}), the operator%
\[
H=\mathfrak{D}^{\alpha}+b\left\Vert x\right\Vert _{p}^{-\alpha}%
\]
is a self-adjoint and non-negative definite operator acting in $L^{2}(X,m)$.
Notice that $b$ is an increasing continuous function of $\beta$ which fulfill
the whole range $\left[  b_{\ast},+\infty\right)  $, where $b_{\ast}%
=-\{\Gamma_{p}\left(  (1+\alpha)/2\right)  \}^{2}$. In particular, $b<0$ for
$(\alpha-1)/2\leq\beta<0$ and $b\geq0$ otherwise.

\begin{theorem}
\label{green_function'}The equation $Hu=v$ has a unique solution
\[
u(x)=%
%TCIMACRO{\dint \limits_{X}}%
%BeginExpansion
{\displaystyle\int\limits_{X}}
%EndExpansion
g_{H}(x,y)v(y)dm(y)\text{,}%
\]
where%
\[
g_{H}(x,y)=h(x)g_{\mathcal{L}}(x,y)h(y).
\]
We call $g_{H}(x,y)$ the Green function of the operator $H$, or the
fundamenthal solution of the equation $Hu=v$.
\end{theorem}

\begin{proof}
We know that $\mathcal{L}:\mathcal{D}\rightarrow L^{2}(X,hm)\cap C_{\infty
}(X)$. Let us show that $\mathcal{L}:\mathcal{D}\rightarrow L^{q}(X,m)$,
$\forall1\leq q\leq\infty$. It is enough to check this property for
$\psi=\mathbf{1}_{B}$, the indicator of an open ball $B$. In this case there
exists a constant $C>0$ such that as $x\rightarrow\infty$ the following
asymptotic relationship holds:%
\begin{align*}
\mathcal{L}\psi(x)  &  =-%
%TCIMACRO{\dint \limits_{B}}%
%BeginExpansion
{\displaystyle\int\limits_{B}}
%EndExpansion
J(x,y)h(y)dm(y)\\
&  =-\frac{1}{\Gamma_{p}(-\alpha)}\frac{1}{\left\Vert x\right\Vert
_{p}^{1+\alpha}}%
%TCIMACRO{\dint \limits_{B}}%
%BeginExpansion
{\displaystyle\int\limits_{B}}
%EndExpansion
hdm\asymp\frac{C}{\left\Vert x\right\Vert _{p}^{1+\alpha}}%
\end{align*}
Clearly this relationship and the fact that $\mathcal{L}\psi(x)$ is bounded
proofs the claim.

In particular, $\mathcal{L}\psi\in L^{2}(X,m)$ and therefore $\frac{1}%
{h}\mathcal{L}\psi\in L^{2}(X,h^{2}m)$ for any $\psi\in\mathcal{D}$. Having
this in mind we do our computations for $\varphi,\psi\in\mathcal{D}:$
\begin{align*}
|Q_{\mathcal{H}}(\varphi,\psi)|  &  =|Q_{\mathcal{L}}(\varphi,\psi
)|=|(\mathcal{L}\psi,\varphi)_{L^{2}(hm)}|\\
&  =\left\vert \left(  \frac{1}{h}\mathcal{L}\psi,\varphi\right)
_{L^{2}(h^{2}m)}\right\vert \leq\left\Vert \frac{1}{h}\mathcal{L}%
\psi\right\Vert _{L^{2}(h^{2}m)}\left\Vert \varphi\right\Vert _{L^{2}(h^{2}%
m)}.
\end{align*}
That means that $\varphi\rightarrow Q_{\mathcal{H}}(\varphi,\psi)$ is a
bounded linear functional in $L^{2}(X,h^{2}m)$ for any $\psi\in\mathcal{D}$.
This fact, in turn, implies that $\mathcal{D}\subset dom(\mathcal{H})$ and
\begin{equation}
\mathcal{H}\psi=\frac{1}{h}\mathcal{L}\psi,\text{ }\forall\psi\in\mathcal{D}.
\label{Eq.H_L}%
\end{equation}
Let us consider the equation $\mathcal{H}u=v$ for $v\in\mathcal{D}$. Since
$\mathcal{D}\subset dom(\mathcal{H})$ we have%
\[
(\mathcal{H}u,\psi)_{L^{2}(X,h^{2}m)}=(u,\mathcal{H}\psi)_{L^{2}(X,h^{2}%
m)},\text{ }\forall\psi\in\mathcal{D}.
\]
Applying equation (\ref{Eq.H_L}) we get
\[
(\mathcal{H}u,\psi)_{L^{2}(X,h^{2}m)}=\left(  u,\frac{1}{h}\mathcal{L}%
\psi\right)  _{L^{2}(X,h^{2}m)}=(u,\mathcal{L}\psi)_{L^{2}(X,hm)}.
\]
On the other hand, we have
\[
(\mathcal{H}u,\psi)_{L^{2}(X,h^{2}m)}=(v,\psi)_{L^{2}(X,h^{2}m)}%
=(hv,\psi)_{L^{2}(X,hm)}.
\]
Our calculations from above show that for H\"{o}lder conjugated $(p,q)$ we
have%
\[
\left\vert (u,\mathcal{L}\psi)_{L^{2}(X,hm)}\right\vert =|(hv,\psi
)_{L^{2}(X,hm)}|\leq\left\Vert hv\right\Vert _{L^{p}(X,hm)}\left\Vert
\psi\right\Vert _{L^{q}(X,hm)}.
\]
It follows that if we choose $1<p<\frac{1+\alpha}{1-\alpha}$, then
$\psi\rightarrow(u,\mathcal{L}\psi)_{L^{2}(X,hm)}$ is a bounded linear
functional in $L^{q}(X,hm)$ provided $q=\frac{p}{p-1}$, i.e. $\frac{1}%
{2}\left(  1+\frac{1}{\alpha}\right)  <q<\infty$.

As $(e^{-t\mathcal{L}})_{t>0}$ is a continuous symmetric Markovian semigroup
an application of the Riesz-Thorin interpolation theorem shows that it can be
extended to all $L^{q}(X,hm)$ as a continuous contraction semigroup. Let
$\mathcal{L}_{q}$ be its minus infinitesimal generator, then $\mathcal{L}_{q}$
extends $\mathcal{L}$, and $\mathcal{L}_{q}^{\ast}=\mathcal{L}_{p}$.

All the above shows that $u$ must belong to the set $dom(\mathcal{L}_{p})$ and
$\mathcal{L}_{p}u=hv$. The equation $\mathcal{L}_{p}u=hv$ has unique solution%
\begin{align*}
u(x)  &  =%
%TCIMACRO{\dint \limits_{X}}%
%BeginExpansion
{\displaystyle\int\limits_{X}}
%EndExpansion
g_{\mathcal{L}}(x,y)(hv)(y)h(y)dm(y)\\
&  =%
%TCIMACRO{\dint \limits_{X}}%
%BeginExpansion
{\displaystyle\int\limits_{X}}
%EndExpansion
g_{\mathcal{L}}(x,y)v(y)h^{2}(y)dm(y).
\end{align*}
It follows that the operator $\mathcal{H}$ acting in $L^{2}(X,h^{2}m)$ admits
a Green function $g_{\mathcal{H}}(x,y)$ and that $g_{\mathcal{H}}(x,y)$
coincides with the function $g_{\mathcal{L}}(x,y)$, the Green function of the
operator $\mathcal{L}$ acting in $L^{2}(X,hm)$:%
\begin{equation}
g_{\mathcal{H}}(x,y)=g_{\mathcal{L}}(x,y). \label{H-L Green functions}%
\end{equation}
Finally, let us consider the equation $Hu=v$. Since $H=U\circ\mathcal{H\circ
}U^{-1}$, we get $\mathcal{H(}U^{-1}u)=U^{-1}v$. It follows that%
\[
\mathcal{(}U^{-1}u)(x)=%
%TCIMACRO{\dint \limits_{X}}%
%BeginExpansion
{\displaystyle\int\limits_{X}}
%EndExpansion
g_{\mathcal{H}}(x,y)(U^{-1}v)(y)h(y)^{2}dm(y),
\]
or equivalently%
\[
u(x)=%
%TCIMACRO{\dint \limits_{X}}%
%BeginExpansion
{\displaystyle\int\limits_{X}}
%EndExpansion
h(x)g_{\mathcal{H}}(x,y)h(y)v(y)dm(y).
\]
That means that equation $Hu=v$ admits a fundamenthal solution
\begin{align*}
g_{H}(x,y)  &  :=h(x)g_{\mathcal{H}}(x,y)h(y)\\
&  =h(x)g_{\mathcal{L}}(x,y)h(y),
\end{align*}
thanks to (\ref{H-L Green functions}). The proof of the theorem is finished.
\end{proof}

\begin{corollary}
\label{g_H Green function}The Green function $g_{H}(x,y)$ is a continuous
function taking finite values off the diagonal set. Moreover, the following
relationship holds:
\begin{equation}
g_{H}(x,y)\asymp\frac{\left\Vert x\right\Vert _{p}^{\beta}\left\Vert
x-y\right\Vert _{p}^{\alpha-1}\left\Vert y\right\Vert _{p}^{\beta}}{\left(
\left\Vert x\right\Vert _{p}\vee\left\Vert y\right\Vert _{p}\right)  ^{2\beta
}},
\end{equation}
or equivalently,%
\[
\frac{g_{H}(x,y)}{g_{L}(x,y)}\asymp\left(  \frac{\left\Vert x\right\Vert _{p}%
}{\left\Vert y\right\Vert _{p}}\wedge\frac{\left\Vert y\right\Vert _{p}%
}{\left\Vert x\right\Vert _{p}}\right)  ^{\beta}.\footnote{It must be compared
with the Green function estimates for Schr\"{o}dinger operators on Riemanian
manifolds, see \cite{Grigoryan}}%
\]

\end{corollary}

\begin{proof}
Follows directly from Theorem \ref{green_function} and Theorem
\ref{green_function'}.
\end{proof}

%\bibliography{ref}

\bigskip

\noindent
A.\ D.\ Bendikov,
Institute of Mathematics, Wroclaw University, Wroclaw, Poland\\
\emph{E-mail address: bendikov@math.uni.wroc.pl}

\bigskip

\noindent
A.\ A.\ Grigor'yan,
Department of Mathematics, University of Bielefeld, Bielefeld, Germany, and
Institute of Control Sciences of Russian Academy of Sciences, Moscow, Russian Federation\\
\emph{E-mail address: grigor@math.uni-bielefeld.de}

\bigskip

\noindent
S.\ A.\ Molchanov,
Department of Mathematics, University of North Carolina, Charlotte, NC 28223,
United States of America, and National Research University, Higher School of
Economics, Russian Federation\\
\emph{E-mail address: smolchan@uncc.edu}
\end{document}